\documentclass[12pt,letterpaper]{article}
\usepackage[left=1.4in,right=1.4in,top=1.5in,bottom=1in]{geometry}

\usepackage[utf8]{inputenc}
\usepackage{amsmath}
\usepackage{amsthm}
\usepackage{amsfonts}
\usepackage{amssymb}
\usepackage{mathtools}
\usepackage{hyperref}
\usepackage{bbm}
\usepackage{algorithm}
\usepackage{algorithmic}
\usepackage{tikz}
\usepackage{booktabs}
\usepackage{bm}
\usepackage{color}
\usepackage[capitalize,noabbrev]{cleveref}

\usepackage{authblk}

\usepackage{graphicx}

\def\R{{\mathbb R}}
\def\N{{\mathbb N}}
\def\Z{{\mathbb Z}}

\newcommand{\M}{\mathcal M}
\newcommand{\rel}{\mathcal R}

\newcommand{\T}{\mathcal T}
\renewcommand{\S}{\mathcal S}
\newcommand{\E}{\mathcal E}
\newcommand{\QP}{\mathcal{BQP}}

\newcommand{\proj}{\operatorname{proj}} % projection
\newcommand{\sign}{\operatorname{sign}}

\newtheorem{observation}{Observation}
\newtheorem{proposition}{Proposition}
\newtheorem{definition}{Definition}
\newtheorem{lemma}{Lemma}
\newtheorem{theorem}{Theorem}
\newtheorem{corollary}{Corollary}

\newcommand{\transp}{\mathsf T}

\def\dim{\mathop{\rm dim}}

\def\int{\mathop{\rm int}}

\def\conv{\mathop{\rm conv}}

\newcounter{claim} %[section]
 \usepackage[usenames,dvipsnames]{pstricks}
 \usepackage{epsfig}
 \usepackage{pst-grad} % For gradients
 \usepackage{pst-plot} % For axes

\makeatletter
\newcommand{\setwindow}[5]{
\def\xmin{#1}%
\def\ymin{#2}%
\def\xmax{#3}%
\def\ymax{#4}%
\pstFPsub\viewingwidth{#3}{#1}%
\pstFPdiv\result{\strip@pt#5}{\viewingwidth}%
\psset{unit=\result pt}}
\makeatother

\author[1]{Alberto Del Pia\thanks{delpia@wisc.edu}}
\author[2]{Jeff Linderoth\thanks{linderoth@wisc.edu}}
\author[3]{Haoran Zhu\thanks{Corresponding author, hzhu94@wisc.edu}} 

\affil[1,2,3]{\small Department of Industrial and Systems Engineering, University of Wisconsin-Madison}
\affil[1]{\small Wisconsin Institute for Discovery, University of Wisconsin-Madison}

\title{Relaxations and Cutting Planes for Linear Programs with Complementarity Constraints}

\date{}

\begin{document}
\maketitle

\begin{abstract}
We study relaxations for linear programs with complementarity constraints, especially instances whose complementary pairs of variables are not independent.  Our formulation is based on identifying vertex covers of the conflict graph of the instance and generalizes the extended reformulation-linearization technique of Nguyen, Richard, and Tawarmalani to instances with general complementarity conditions between variables.  We demonstrate how to obtain strong cutting planes for our formulation from both the stable set polytope and the boolean quadric polytope associated with a complete bipartite graph. Through an extensive computational study for three types of practical problems, we assess the performance of our proposed linear relaxation and new cutting-planes in terms of the optimality gap closed.

\end{abstract}

\section{Introduction}
\label{sec: intro}

In this paper we study the \emph{linear programming problem with complementarity constraints} (LPCC) of the following form:
\begin{align}
\label{LPCC}
\tag{LPCC}
\begin{split} 
\text{maximize} \quad & \pmb{a}^\transp \pmb{x} + \pmb{b}^\transp \pmb{y}\\
\text{subject to} \quad & A \textbf{\textit{x}} + B \textbf{\textit{y}} \le \textbf{\textit{d}}, \\
& 0 \leq \textbf{\textit{y}} \leq \textbf{1}_n, \\
&  y_i \cdot y_j = 0 \quad \forall \{i,j\} \in E.
\end{split}
\end{align}
Here $A \in \R^{m \times p}, B \in \R^{m \times n}, \textbf{\textit{a}} \in \R^p, \textbf{\textit{b}} \in \R^n, \textbf{\textit{d}} \in \R^m$ for $m,n,p \in \N$, and $\textbf{1}_n$ denotes the $n$-dimensional vector with all components having value one. We assume that finite upper bounds on the variables \textbf{\textit{y}} are given, so it is without loss of generality that we may scale their upper bounds to be $\textbf{1}_n$. 
The set $E$ consists of unordered pairs of elements in $[n] := \{1, \ldots, n\}$, and we denote by $G = ([n], E)$ the so-called \emph{conflict graph} of \eqref{LPCC}.
Throughout the paper we assume that $G$ is simple, so it contains no loops or parallel edges.
We denote by $P^\perp$ the feasible region of \eqref{LPCC}, and by 
$$
P = \{(\textbf{\textit{x}}, \textbf{\textit{y}}) \in \R^{p+n} \mid A \textbf{\textit{x}} + B \textbf{\textit{y}} \leq \textbf{\textit{d}}, 0 \leq \textbf{\textit{y}} \leq \textbf{1}_n\}
$$
the \emph{linear programming (LP) relaxation} of $P^\perp$.

The conditions $y_i \cdot y_j = 0, \{i,j\} \in E,$ known as \emph{complementarity constraints}, are a special case of \emph{Special Ordered Sets of type 1} (SOS1) \cite{beale1970special}, which is a set of variables in which at most one may take a non-zero value.
Practical problems involving complementarity constraints arise extensively in business, engineering, and economics, and we refer the reader to \cite{ferris1997engineering} for a detailed survey of applications. 
In many applications of LPCC, its associated conflict graph is \emph{1-regular}, i.e., every variable $y_i, i \in [n]$ is contained in exactly one complementarity constraint. 
An important example of this case arises from imposing optimality conditions on decisions in a multi-level setting.  In other applications, however, the conflict graph $G$ is more general.  For example, consider scheduling a set of jobs $[n]$, where each job $i \in [n]$ has processing times $a_i$ and benefit $c_i$. The goal is to optimally select a subset $I \subseteq [n]$ of jobs to complete in a given time limit $b$, where jobs may be partially executed, and certain pairs of jobs may not both be selected.  This problem, known as the \emph{continuous knapsack problem with conflicts}, can be formulated as an LPCC:
$$
\max \big\{\textbf{\textit{c}}^\transp \textbf{\textit{y}} \mid \textbf{\textit{a}}^\transp \textbf{\textit{y}} \leq b, 0 \leq \textbf{\textit{y}} \leq \textbf{1}_n, \ y_i \cdot y_j = 0 \ \forall \{i,j\} \in E \big\}.
$$

In this paper, we propose a new extended formulation for $\conv(P^\perp)$ whose natural relaxation is strong, and we demonstrate additional cutting-planes to strengthen the relaxation. We call a set $\mathcal R$ an \emph{extended relaxation} for a set $X$ in $\R^n$ if $\dim(\mathcal R)>n$ and $X \subseteq \proj_{\R^n} \mathcal R$.  
A recent paper by Nguyen et al. \cite{MR4270189} is closely related to our research.  The authors propose an extension of the well-known reformulation-linearization technique (RLT) of Sherali and Adams \cite{MR1061981}, coined ERLT, to the LPCC with a 1-regular conflict graph, and show that the resulting extended relaxation is stronger.  As we will see in Section \ref{sec: 1-r}, when the conflict graph is 1-regular, our extended relaxation coincides with the relaxation given by the first-level ERLT.  The paper \cite{MR4270189} also describes the standard, iterative, multi-level process to strengthen ERLT.  From a computational perspective, the first level ERLT relaxation is the most relevant, and throughout this paper, when we reference the ERLT relaxation, we mean the first-level ERLT.

For \eqref{LPCC} with a 1-regular conflict graph, many approaches have been given for strengthening the LP relaxation. 
Ibaraki \cite{ibaraki1973use} applies the work of Balas \cite{MR290793} to obtain a family of cutting-planes that can be derived from the simplex tableau of the linear programming relaxation of LPCC, and an improvement to these cuts is proposed by Sherali et al. \cite{MR1397652} using disjunctive programming techniques. 
Jeroslow \cite{jeroslow1978cutting} gave a finite sequential procedure for generating all valid inequalities of the feasible set when every variable appears in some complementarity constraint. Moreover, Hu et al. \cite{MR2403040, MR2914070} studied the complementarity constraints arising from KKT conditions for linear programming problems with quadratic objective, and proposed a Benders type approach for solving the resulting LPCC. 
When the complementarity constraints are a collection of disjoint SOS1, the corresponding conflict graph is the union of a set of cliques. In this case, de Farias et al. \cite{de2001branch, de2002facets, de2014branch} analyzed the polyhedral properties of the feasible sets that arise in this scenario and derived cutting-planes by a sequential lifting procedure of cover inequalities. Moreover, with the setting, Del Pia et al. \cite{del2022new} studied how to derive strong cutting-planes from another complementary concept of cover called ``pack".
Most existing research for the case of a general conflict graphs focuses on integer programming \cite{MR3530661, MR1780979, MR1795056, hifi2007reduction, MR4330950, pferschy2009knapsack}. Recently, Fischer and Pfetsch \cite{MR3773088} investigated a branch-and-cut algorithm to solve LPCC with a general conflict graph and demonstrated its effectiveness by an extensive computational study.

The structure of this paper is as follows: In Section \ref{sec: ext_relax}, using concepts of vertex cover in graph theory and motivated by the disjunctive formulation, we present an extended relaxation for $\conv(P^\perp)$. We show that for the special case where the conflict graph $G$ is 1-regular, our extended relaxation coincides with the ERLT relaxation introduced by Nguyen et al. \cite{MR4270189}. 
In particular, from the disjunctive perspective of our formulation and the idealness of Balas' disjunctive formulation, we are able to derive Theorem 2 and Theorem 3 in \cite{MR4270189}. 
In Section~\ref{sec: cutting_plane}, we study a sub-structure that arises from our proposed extended formulation, and we show that all valid inequalities for a particular \emph{Boolean Quadric Polytope} (BQP) \cite{MR1017216} associated with a complete bipartite graph, can be added as cutting-planes to our extended relaxation proposed in the previous section. 
In Section~\ref{sec: experiments}, we conduct an extensive numerical study over a series of randomly generated problems and illustrate that when combined with the cutting-planes from Section~\ref{sec: cutting_plane}, our extended relaxation is able to significantly close the optimality gap, compared to some other benchmark linear relaxations. 
We give concluding remarks in Section~\ref{sec: conclude}.

\textit{Notation.} 
Throughout this paper, vectors are written in bold, e.g., $\textbf{\textit{x}} = (x_1, \ldots, x_p)^\transp$, and matrices are written in capital letters.
The notation $M_i$ is used to denote the $i$-th row of a matrix $M \in \R^{m \times n}$, for $i \in [m]$.  For a given vector $\textbf{\textit{x}} \in \R^n$ and set $S \subseteq [n]$, we use the common notation $\textbf{\textit{x}}(S): = \sum_{i \in S} x_i$.  For a node $i \in [n]$ of the conflict graph $G$, the notation $\delta(i)$ refers to the set of nodes in $G$ adjacent to $i$, and for any $T \subseteq [n]$, we define $\delta(T): = \cup_{i \in T}\delta(i)$. 
The notation $\textbf{1}_n$ denotes the $n$-dimensional vector with all components 1, and the orthogonal projection of a set $X$ onto the space of $\textbf{\textit{x}}$-variables is given by $\proj_\textbf{\textit{x}} X.$  The set of all subsets of a given set $S$ is denoted as $2^S$.

\section{A Vertex cover-based extended relaxation}
\label{sec: ext_relax}
In this section, we introduce a new extended relaxation for \eqref{LPCC} that generalizes the ERLT procedure \cite{MR4270189} to the case of complementarity conditions forming an arbitrary conflict graph.  The extension relies on the concept of vertex covers.

%\begin{definition}
%A \emph{dominating set} of a graph $G = (V, E)$ is a subset $D$ of $V$ such that every vertex not in $D$ is adjacent to at least one member of $D$.
%\end{definition}
%Then a \emph{minimum dominating set} of a graph $G$ is defined as the dominating set of $G$ with the smallest size, and its size is normally denoted as $\gamma(G)$, which is called \emph{dominating number} of $G$. When it is clear from the context, we simply refer to it as $\gamma$. Throughout this paper, for any node $i$ of a graph $G$, we denote $\delta(i)$ as the set of adjacent nodes of $i$ in $G$.
\begin{definition}
A \emph{vertex cover} of a graph $G = ([n], E)$ is a subset $C$ of $[n]$ such that for every edge $\{i,j\} \in E$, at least one among $i$ and $j$ is in $C$. A \emph{minimum vertex cover} is a vertex cover with the smallest size. 
\end{definition}
%A \emph{minimum vertex cover} is a vertex cover of smallest possible size, and the \emph{vertex cover number} $\tau$ is the size of such minimum vertex cover. 

The following lemma follows directly from the definition of a vertex cover.
\begin{lemma}
\label{lem: reform_comp}
Let $\textbf{\textit{y}} \in \R^n$ and let $C$ be a vertex cover of $G=([n], E)$. 
Then $\textbf{\textit{y}}$ satisfies the complementarity constraints of \eqref{LPCC} if and only if for any $i \in C$, either $y_i = 0$ or $y_j = 0$ for every $j \in \delta(i)$.
\end{lemma}

\begin{proof}
First, assume $y$ satisfies the complementarity constraints $y_i \cdot y_j = 0$ for all $\{i,j\} \in E$, and assume that $y_{i'} \neq 0$ for some $i' \in C$. 
Then we must have $y_j = 0$ for any $j \in \delta(i')$. 

On the other hand, assume that for any $i \in C$, either $y_i = 0$ or $y_j = 0$ for every $j \in \delta(i)$.  Take an arbitrary edge $\{i', j'\} \in E$.  Since $C$ is a vertex cover, either $i' \in C$ or $j' \in C$.  Suppose without loss of generality that $i' \in C$. Then by assumption we obtain that either $y_{i'} = 0$ or $y_j = 0$ for every $j \in \delta(i')$. Since $j' \in \delta(i')$, it follows that $y_{i'} \cdot y_{j'} = 0$.
\end{proof}

Our extended relaxation is based on the notion of a (feasible) cover partition.
\begin{definition}
We say that $\T \subseteq 2^{[n]}$ is a \emph{cover partition} of $G$ if $\cup_{T \in \T} T$ is a vertex cover of $G$, and $T_1 \cap T_2 = \emptyset$ for any $T_1, T_2 \in \T$.  A cover partition $\T$ of $G$ is a \emph{feasible cover partition} if for any $T \in \T$ and any $i \in T$, we have $\delta(i) = \delta(T)$. 
\end{definition}

For notational convenience, for any set $T \subseteq [n]$, define the sets
\begin{align*}
    H_T &:= \{(\textbf{\textit{x}}, \textbf{\textit{y}}) \in \R^{p+n} \mid y_i = 0 \ \forall i \in T\} \mbox { and}\\
    \bar H_T &:= \{(\textbf{\textit{x}}, \textbf{\textit{y}}) \in \R^{p+n} \mid y_j = 0 \ \forall  j \in \delta(T)\}.
\end{align*}
From Section~\ref{lem: reform_comp}, it follows that if $C$ is a vertex cover of the conflict graph $G$ of \eqref{LPCC}, then $P^\perp = P \cap_{i \in C} (H_{\{i\}} \cup \bar H_{\{i\}})$.  If $\T$ is a feasible cover partition of $G$, then because for any $T \in \T$ and $i,j \in T, \bar H_i = \bar H_j = \bar H_T$, we have that $\cap_{i \in C} (H_{\{i\}} \cup \bar H_{\{i\}}) = \cap_{T \in \T} (H_T \cup \bar H_T)$. Thus, we have the following chain of equalities:
\begin{equation}
\label{eq: reform_Pperp}
P^\perp = P \cap_{i \in C} (H_{\{i\}} \cup \bar H_{\{i\}})  = P \cap_{T \in \T} (H_T \cup \bar H_T) = \bigcap_{T \in \T} \big( (P \cap H_T) \cup (P \cap \bar H_T)  \big).
\end{equation}
Hence for a feasible cover partition $\T$, one natural convex relaxation for $P^\perp$ is
\begin{equation*}
%\label{eq: relax_origin}
\bigcap_{T \in \T} \conv\big( (P \cap H_T) \cup (P \cap \bar H_T)  \big).
\end{equation*}
In the remainder of the paper, we always assume that $\T$ is a feasible cover partition.  We first notes that this relaxation is a polyhedron.

\begin{observation}
\label{obs: easy_disjun_polyhedra}
For any $P \subseteq \R^{p} \times [0,1]^n$ and $T \subseteq [n], \conv\big( (P \cap H_T) \cup (P \cap \bar H_T)  \big)$ is a polyhedron.
\end{observation}

\begin{proof}
Since $P \subseteq \R^{p} \times [0,1]^n$ and the hyperplanes defining $H_T$ and $\bar H_T$ only involve $y$ variables, the recession cone of the both polyhedra $P \cap H_T$ and $P \cap \bar H_T$ is $\{(\textbf{\textit{x}}, \pmb{0}) \mid A \textbf{\textit{x}} \leq \pmb{0}\}$.
It is well-known (see, e.g., \cite{Conforti:2014:IP:2765770}) that if the polyhedra share the same recession cone, the convex hull of the union of the polyhedra is polyhedral.
\end{proof}

A consequence of Section~\ref{obs: easy_disjun_polyhedra} is that (unlike in \cite{MR4270189}), we need not assume $P^\perp$ is compact. 
From Section~\ref{obs: easy_disjun_polyhedra} and Balas' classical results on disjunctive programming \cite{MR1663099, MR791175}, for each $T \subseteq [n]$, $\conv\big( (P \cap H_T) \cup (P \cap \bar H_T)  \big)$ has the following extended formulation:
\begin{align}
A \pmb{u_T} + B \pmb{v_T}  \leq \textbf{\textit{d}}  q_T \label{eq: balas 1} \\ 
A \pmb{\bar u_T} + B \pmb{\bar v_T}  \leq \textbf{\textit{d}}  \bar{q}_T \label{eq: balas 2}\\
\bar{v}_{T,i} = 0 \ \forall i \in T, \ v_{T,j} = 0 \ \forall j \in \delta(T) \label{eq: balas 3}\\
0 \leq \pmb{v_T} \leq q_T \cdot \textbf{1}_n,  \ 0 \leq \pmb{\bar v_T} \leq \bar q_T \cdot \textbf{1}_n \label{eq: balas 4}\\
\pmb{\bar u_T} + \pmb{u_T} = \textbf{\textit{x}}, \ \pmb{\bar v_T} + \pmb{v_T} = \textbf{\textit{y}} \label{eq: balas 5}\\
\bar{q}_T + q_T = 1, \ q_T, \bar{q}_T \geq 0, \label{eq: balas 6}
\end{align}
where $\pmb{u_T} \in \R^p$, $\pmb{\bar u_T} \in \R^p$, $\pmb{v_T} \in \R^n$, $\pmb{\bar v_T} \in \R^n$ are vectors of variables, and $q_T$ and $\bar{q}_T$ are scalar variables. 
% $\pmb{u_T} = (u_{T, 1}, \ldots, u_{T, p})^\transp \in \R^p,$ $\pmb{\bar u_T} = (\bar u_{T, 1}, \ldots, \bar u_{T, p})^\transp \in \R^p,$ $\pmb{v_T} = (v_{T, 1}, \ldots, v_{T, n})^\transp \in \R^n,$ $\pmb{\bar v_T} = (\bar v_{T, 1}, \ldots, \bar v_{T, n})^\transp \in \R^n$.
For a given feasible cover partition $\T$, by enumerating the extended formulations for each $T \in \T$, we obtain a relaxation for $P^\perp$ in the extended space $\R^{p+n+2 |\T| + 2|\T| p + 2|\T| n}$.  We call this extended relaxation the \emph{vertex-cover-based extended relaxation} of $P^\perp$:
\begin{equation*}
\E_{\T}(P) := \big\{(\textbf{\textit{x}},\textbf{\textit{y}},\textbf{\textit{q}}, \bar{\textbf{\textit{q}}}, \textbf{\textit{u}},\bar{\textbf{\textit{u}}}, \textbf{\textit{v}}, \bar{\textbf{\textit{v}}}) \in \R^{p+n+2|\T| + 2|\T| p + 2|\T| n} \mid  \eqref{eq: balas 1}-\eqref{eq: balas 6} \ \forall T \in \T \big\}.
\end{equation*}

Here we should remark that $\E_{\T}(P)$ only depends on the linear constraints of the polyhedron $P$ and the feasible cover partition $\T$.
Due to the natural connection between $\E_{\T}(P)$ and the disjunctive formulation of $\conv\big( (P \cap H_T) \cup (P \cap \bar H_T)  \big)$ with $T \in \T$, we naturally have the following result.

\begin{theorem}
\label{theo: balas_disjun_nice}
Let $Q \subseteq \R^{p} \times [0,1]^n$ be a polyhedron such that $P^\perp \subseteq Q$, and let $\T$ be a feasible cover partition of the conflict graph $G$. Then for each $T \in \T$,
$$\conv(P^\perp) \subseteq \proj_{\textbf{\textit{x}},\textbf{\textit{y}}} \E_{\T}(Q) \subseteq \conv\big( (Q \cap H_T) \cup (Q \cap \bar H_T)  \big).$$
\end{theorem}

\begin{proof}
The set $\E_{\T}(Q)$ is constructed by enumerating Balas' disjunctive formulation of $\conv\big( (Q \cap H_T) \cup (Q \cap \bar H_T)  \big)$ for each $T \in \T$, so the second subset relation holds trivially. 
To show the first subset relation, since $\proj_{\textbf{\textit{x}},\textbf{\textit{y}}} \E_{\T}(Q)$ is a convex set, we need only need to show $P^\perp \subseteq \proj_{\textbf{\textit{x}},\textbf{\textit{y}}} \E_{\T}(Q)$, so we will show that for any $(\textbf{\textit{x}}, \textbf{\textit{y}}) \in P^\perp$, there exists $(\textbf{\textit{q}}, \bar{\textbf{\textit{q}}}, \textbf{\textit{u}}, \bar{\textbf{\textit{u}}}, \textbf{\textit{v}}, \bar{\textbf{\textit{v}}}) \in \R^{2|\T| + 2|\T| p + 2|\T|n}$ such that $(\textbf{\textit{x}}, \textbf{\textit{y}}, \textbf{\textit{q}}, \bar{\textbf{\textit{q}}}, \textbf{\textit{u}}, \bar{\textbf{\textit{u}}}, \textbf{\textit{v}}, \bar{\textbf{\textit{v}}}) \in \E_{\T}(Q)$. 
Given $(\textbf{\textit{x}}, \textbf{\textit{y}}) \in P^\perp$, for each $T \in \T$, let
\[
q_T := 
\begin{cases}
0 & \text {if } \textbf{\textit{y}}(T) +  \textbf{\textit{y}}(\delta(T)) = 0 \\
\frac{\textbf{\textit{y}}(T)}{\textbf{\textit{y}}(T) + \textbf{\textit{y}}(\delta(T)) } & \text{otherwise,}
\end{cases}
\qquad
\]
$\bar{q}_T := 1 - q_T, \textbf{\textit{u}}_{T} := q_T \cdot \textbf{\textit{x}}, \bar{\textbf{\textit{u}}}_{T} := \bar q_T  \cdot \textbf{\textit{x}}, \textbf{\textit{v}}_{T} := q_T  \cdot \textbf{\textit{y}}, \bar{\textbf{\textit{v}}}_{T} := \bar q_T  \cdot \textbf{\textit{y}}$.
Since $(\textbf{\textit{x}}, \textbf{\textit{y}}) \in P^\perp \subseteq Q$, by the definition of $(\textbf{\textit{u}}, \bar{\textbf{\textit{u}}}, \textbf{\textit{v}}, \bar{\textbf{\textit{v}}})$, the corresponding constraints \eqref{eq: balas 1} and \eqref{eq: balas 2} with respect to $Q$ are naturally satisfied. Now we check constraint \eqref{eq: balas 3}.  Since $\T$ is a feasible cover partition, for any $T \in \T$ and $i \in T$, $\delta(i) = \delta(T)$  
Also, since $\textbf{\textit{y}}$ satisfies the complementarity constraints, for any $i \in T$, either $y_i = 0$ or $\textbf{\textit{y}}(\delta(i)) = \textbf{\textit{y}}(\delta(T)) = 0$. 
Hence, either $\textbf{\textit{y}}(T) = 0$ or $\textbf{\textit{y}}(\delta(T)) = 0$.
In the first case, $q_T$ is defined to be 0, and equations of \eqref{eq: balas 3} hold; In the second case, $q_T = \mathbbm{1}\{\textbf{\textit{y}}(T) > 0\}$ and \eqref{eq: balas 3} holds. 
It is easy to verify that the remaining constraints \eqref{eq: balas 4}-\eqref{eq: balas 6} are also satisfied by $(\textbf{\textit{x}}, \textbf{\textit{y}}, \textbf{\textit{q}}, \bar{\textbf{\textit{q}}}, \textbf{\textit{u}}, \bar{\textbf{\textit{u}}}, \textbf{\textit{v}}, \bar{\textbf{\textit{v}}})$ for any $T \in \T$. 
\end{proof}

We define $\rel_{\T}(P) := \proj_{\textbf{\textit{x}},\textbf{\textit{y}}} \E_{\T}(P)$ to be the linear relaxation of $P^\perp$ in the original space of variables. 
Furthermore, for any $t \in \N$ and $t \geq 2$, we recursively define $\rel_{\T}^t(P) = \rel_{\T}(\rel_{\T}^{t-1}(P))$, where $\rel_{\T}^1(P): = \rel_{\T}(P)$. 
We obtain the following theorem.

\begin{theorem}
\label{theo: balas_disjun_nice_recur}
Let $\T$ be a feasible cover partition of $G$. For any $k \in [|\T|]$, and for any collection of $k$ distinct elements $T_1, \ldots, T_k$ of $\T$,
\begin{equation}
\label{eq: recur}
\conv(P^\perp) \subseteq \rel_{\T}^k(P) \subseteq \conv\big( \bigcap_{j=1}^k (P \cap H_{T_j}) \cup (P \cap \bar H_{T_j})  \big).
\end{equation}
%In particular, $\rel_{T}^t(P) = \conv(P^\perp)$. 
\end{theorem}

\begin{proof}
The proof is by induction. 
When $k=1$, this result reduces to Section~\ref{theo: balas_disjun_nice} by picking $Q=P$. Assume now that the result holds for $k \leq \kappa-1$, where $\kappa \in [|\T|]$ and consider the case $k=\kappa$.
Let $Q = \rel_{\T}^{\kappa-1}(P)$, and $H = \cap_{j=1}^{\kappa-1} (H_{T_j} \cup \bar H_{T_j})$.
By the inductive hypothesis, we have:
\begin{equation}
\label{eq: tmp1}
\conv(P^\perp) \subseteq Q \subseteq \conv\big( \bigcap_{j=1}^{\kappa-1} (P \cap H_{T_j}) \cup (P \cap \bar H_{T_j})  \big) = \conv(P \cap H).
\end{equation}
Applying Section~\ref{theo: balas_disjun_nice} for $Q$ and $T_\kappa \in \T$, we have
\begin{equation}
\label{eq: tmp2}
\conv(P^\perp) \subseteq \rel_{\T}(Q) \subseteq \conv\big( (Q \cap H_{T_\kappa}) \cup (Q \cap \bar H_{T_\kappa})  \big).
\end{equation}
Combining \eqref{eq: tmp1} and \eqref{eq: tmp2}, we obtain
\begin{align*}
\conv(P^\perp) & \subseteq \rel_{\T}^{\kappa}(P) = \rel_{\T}(Q) \\
& \subseteq \conv\big( (\conv(P \cap H) \cap H_{T_\kappa}) \cup (\conv(P \cap H) \cap \bar H_{T_\kappa})  \big).
\end{align*}
Notice that both $H_{T_\kappa} \cap \conv(P \cap H)$ and $\bar H_{T_\kappa} \cap \conv(P \cap H)$ are faces of the set $\cap \conv(P \cap H)$, hence we have $\conv(P \cap H) \cap H_{T_\kappa}= \conv(P \cap H \cap H_{T_\kappa})$ and $\conv(P \cap H) \cap \bar H_{T_\kappa}= \conv(P \cap H \cap \bar H_{T_\kappa})$. 
Therefore:
\begin{align*}
\conv(P^\perp)  \subseteq \rel_{\T}^{\kappa}(P) & \subseteq \conv\big( \conv(P \cap H \cap H_{T_\kappa}) \cup \conv(P \cap H \cap \bar H_{T_\kappa})  \big) \\
& = \conv \big( (P \cap H \cap H_{T_\kappa})  \cup (P \cap H \cap \bar H_{T_\kappa}) \big) \\
& = \conv \big( \bigcap_{j=1}^{\kappa} (P \cap H_{T_j}) \cup (P \cap \bar H_{T_j})  \big).
\end{align*}
Hence \eqref{eq: recur} also holds for $k=\kappa$, and the proof is complete. 
\end{proof}
Note that Section~\ref{theo: balas_disjun_nice_recur} implies that $\rel_{\T}^{|\T|}(P) = \conv(P^\perp)$ for any feasible cover partition $\T$ of $G$, which naturally leads to the following definition.
\begin{definition}
\label{defn: rank}
Let $\T$ be a feasible cover partition for the conflict graph $G$ of \eqref{LPCC} whose feasible region is $P^\perp$.  The \emph{cover-partition rank} of $\T$ for $P^\perp$, denoted $r_{\T}(P^\perp)$, is the least number $r$ such that $\rel_{\T}^{r}(P) = \conv(P^\perp)$.
\end{definition}
An immediate corollary of Section~\ref{theo: balas_disjun_nice_recur} is that the cover-partition rank is upper bounded by the cardinality of the feasible cover partition for any polyhedron $P$.
\begin{corollary}
\label{cor: rank_ub}
$r_{\T}(P) \leq |\T|$.
\end{corollary}

\subsection{Edge-based extended relaxation}
\label{sec: relaxation_edge_by_edge}

%From the same disjunctive programming perspective, there are various other ways of construction the extended relaxation for $\conv(P^\perp)$. Based on the chain of equation~\eqref{eq: reform_Pperp}, one obvious question is: how about the relaxation based on formulating $\conv\big( P \cap H_i) \cup (P \cap  \bar H_i) \big)$ term by term? In other words, why not simply pick the feasible cover partition as $\big\{ \{i\} \ \forall i \in [n]\big\}$? 
%Follow the same formulation idea, here we have a different extended relaxation from $\E_{\T}(P)$:
%\begin{equation}
%\label{eq: all_nodes}
%\begin{split}
%A u_{i,\cdot} + B v_{i,\cdot}  \leq d \cdot q_i, & \quad \forall i \in [n] \\
%A \bar{u}_{i,\cdot} + B \bar{v}_{i,\cdot}  \leq d \cdot \bar{q}_i, & \quad \forall i \in [n] \\
%\bar{v}_{i,i} = 0, \ v_{e,j} = 0, & \quad \forall i \in [n] \\
%0 \leq v_{e, \cdot} \leq q_e \cdot \textbf{1}_n,  \ 0 \leq \bar v_{e, \cdot} \leq \bar q_e \cdot \textbf{1}_n, & \quad \forall i \in [n] \\
%\bar{u}_{e, \cdot} + u_{e, \cdot} = x, \ \bar{v}_{e, \cdot} + v_{e, \cdot} = y, & \quad \forall i \in [n] \\
%\bar{q}_e + q_e = 1, \ q_e, \bar{q}_e \geq 0, & \quad \forall i \in [n].
%\end{split}
%\end{equation}

From a similar disjunctive perspective as that which motivated the chain of equations~\eqref{eq: reform_Pperp}, a natural relaxation of $\conv(P^\perp)$ contain be obtained using the following relations:

$$
\conv(P^\perp) = \conv \big(\bigcap_{ \{i,j\} \in E} (P \cap H_i) \cup (P \cap  H_j)  \big) \subseteq \bigcap_{\{i,j\} \in E} \conv\big( (P \cap H_i) \cup (P \cap H_j)  \big).
$$
Therefore, using disjunctive programming, the following linear system provides another extended relaxation of $P^\perp$:
\begin{equation}
\label{eq: systemM}
\begin{split}
A \pmb{u_e} + B \pmb{v_e}  \leq \textbf{\textit{d}}  q_e, & \quad \forall e \in E \\
A \pmb{\bar u_e} + B \pmb{\bar v_e}  \leq \textbf{\textit{d}}  \bar{q}_e, & \quad \forall e \in E \\
\bar{v}_{e,i} = 0, \ v_{e,j} = 0, & \quad \forall e=\{i,j\} \in E \\
0 \leq \pmb{v_e} \leq q_e \cdot \textbf{1}_n,  \ 0 \leq \pmb{\bar v_e} \leq \bar q_e \cdot \textbf{1}_n, & \quad \forall e \in E \\
\pmb{\bar u_e} + \pmb{u_e} = \textbf{\textit{x}}, \ \pmb{\bar v_e} + \pmb{v_e} = \textbf{\textit{y}}, & \quad \forall e \in E \\
\bar{q}_e + q_e = 1, \ q_e, \bar{q}_e \geq 0, & \quad \forall e \in E.
\end{split}
\end{equation}

Let $\M$ denote the set in $\R^{p+n+2|E|+2|E|p + 2|E|n}$ given by the linear system \eqref{eq: systemM}. 
It can be shown that $\M$ is never a stronger relaxation than the relaxation provided by $\E_\T(P)$ with respect to some feasible cover partition $\T$.

\begin{proposition}
\label{prop: comparison}
Let $\T^* = \{ \{i\} \mid \{i,j\} \in E \}$.
Then $ \rel_{\T^*}(P) \subseteq \proj_{\textbf{\textit{x}},\textbf{\textit{y}}} \M$.
\end{proposition}

\begin{proof}
Arbitrarily pick $(\pmb{x^*}, \pmb{y^*}) \in \rel_{\T^*}(P) = \proj_{\textbf{\textit{x}},\textbf{\textit{y}}} \E_{\T^*}(P)$.  There exist variables $(\pmb{q^*}, \pmb{\bar q^*}, \pmb{u^*}, \pmb{\bar u^*}, \pmb{v^*}, \pmb{\bar v^*})$ such that $(\pmb{x^*}, \pmb{y^*}, \pmb{q^*}, \pmb{\bar q^*}, \pmb{u^*}, \pmb{\bar u^*}, \pmb{v^*}, \pmb{\bar v^*})$ satisfies the following constraints:
\begin{equation}
\label{eq: another_disjunction}
\begin{split}
A \pmb{u^*_T} + B \pmb{v^*_T} \leq \textbf{\textit{d}}  q^*_T, & \quad \forall T \in \T^* \\
A \pmb{\bar u^*_T} + B \pmb{\bar v^*_T}  \leq \textbf{\textit{d}}  \bar{q}^*_T, & \quad \forall T \in \T^* \\
\bar{v}^*_{T,i} = 0, \ v^*_{T,j} = 0 \ \forall j \in \delta(i), & \quad \forall T = \{i\} \in \T^* \\
0 \leq \pmb{v^*_T} \leq q^*_T \cdot \textbf{1}_n,  \ 0 \leq \pmb{\bar v^*_T} \leq \bar q^*_T \cdot \textbf{1}_n, & \quad \forall T \in \T^* \\
\pmb{\bar u^*_T} + \pmb{u^*_T} = \textbf{\textit{x}}^*, \ \pmb{\bar v^*_T} + \pmb{v^*_T} = \textbf{\textit{y}}^*, & \quad \forall T \in \T^* \\
\bar{q}^*_T + q^*_T = 1, \ q^*_T, \bar{q}^*_T \geq 0, & \quad \forall T \in \T^*. 
\end{split}
\end{equation}

Now, for any $e=\{i,j\} \in E$, we construct:
$$
q^\circ_e := q^*_{\{i\}}, \ \bar q^\circ_e := \bar q^*_{\{i\}}, \ \pmb{u^\circ_e} := \pmb{u^*_{\{i\}}}, \ \pmb{\bar u^\circ_e} := \pmb{\bar u^*_{\{i\}}}, \ \pmb{v^\circ_e} := \pmb{v^*_{\{i\}}}, \ \pmb{\bar v^\circ_e} := \pmb{\bar v^*_{\{i\}}}.
$$
Given~\eqref{eq: another_disjunction} is simple to check that $(\pmb{x^*}, \pmb{y^*}, \pmb{q^\circ}, \pmb{\bar q^\circ}, \pmb{u^\circ}, \pmb{\bar u^\circ}, \pmb{v^\circ}, \pmb{\bar v^\circ})$ satisfies all the constraints \eqref{eq: systemM}. 
Thus $(\pmb{x^*}, \pmb{y^*}) \in \proj_{\textbf{\textit{x}},\textbf{\textit{y}}} \M$. 
%By the arbitrariness of $(\pmb{x^*}, \pmb{y^*}) \in \rel_{\T^*}(P)$, we complete the proof. 
\end{proof}

Later in Section~\ref{sec: experiments}, we will show empirically that $\E_{\T}(P)$ can provide a significantly stronger relaxation than $\mathcal M$.  
The strength of the relaxation depends on the choice of $\T$, so an important question is how to select the cover partition.  Intuition given by Section~\ref{cor: rank_ub} suggests that we should choose a feasible cover partition of small cardinality, as the cardinality upper bounds the cover-partition rank.  In this case, the simple choice of feasible cover partition $\{ \{i\} \mid \{i,j\} \in E\}$ of Proposition~\ref{prop: comparison} is not likely to lead to the strongest relaxation.   
Moreover, a feasible cover partition $\T$ with cardinality $t$ produces a linear relaxation $\E_{\T}(P)$ with $O(mt+nt+pt)$ linear constraints and $(2+2p+2n)t$ additional variables, so a smaller feasible cover partition will also be more favorable from a computational perspective. 
In the numerical experiments that we present later in Section~\ref{sec: experiments}, we will use the feasible cover partition given by some approximate minimum vertex cover, which can be easily accessed using the Python package \emph{NetworkX}.

\subsection{Equivalence of vertex cover-based extended relaxation and ERLT for 1-regular conflict graph}
\label{sec: 1-r}
In this section, we demonstrate that if the conflict graph of complementarity condition $G$ is 1-regular, then the vertex-cover-based extended relaxation $\E_\T(P)$ is the same as the relaxation from the Extended Reformulation Linearization Technique (ERLT) proposed by Nguyen et al. \cite{MR4270189}. 

We begin by briefly reviewing the ERLT procedure.
Let $Q = \{(\pmb{x}, \pmb{y}, \pmb{z}) \in \R^{p+2n} \mid \bar A \pmb{x} + \bar B \pmb{y} + \bar C \pmb{z} \leq \pmb{\bar d}\}$ be a given polytope in $\R^n_+$, where the complementarity constraints are $y_j \cdot z_j = 0$ for any $j \in [n]$. Given a matrix $ F \in \R^{ m \times  n}$ and a subset $K \subseteq [n]$, we denote by $F^K$ the sub-matrix obtained by removing from $F$ all columns with indices not in $K$, and denote by $F^{-K}$ the sub-matrix obtained by removing from $F$ all columns with indices in $K$. Using this notation, then applying the ERLT to $Q$ gives the linear description
\begin{equation}
    \label{eq: erlt_formulation}
    \begin{split}
        & \bar{A} \pmb{u_j} + \bar{B}^{-\{j\}} \pmb{v_j} + \bar{C}^{-\{j\}} \pmb{w_j} + \bar{B}^{\{j\}} y_j  \leq \bar{\textbf{\textit{d}}} q_j, \ \forall j \in N,\\
        & \bar{A} (\pmb{x}-\pmb{u_j}) + \bar{B}^{-\{j\}} (\pmb{y^{-\{j\}}} - \pmb{v_j}) + \bar{C}^{-\{j\}} (\pmb{z^{-\{j\}}} - \pmb{w_j}) + \bar{C}^{\{j\}} z_j  \leq \bar{\textbf{\textit{d}}} (1-q_j), \ \forall j \in N,
    \end{split}
\end{equation}
in the variables $\pmb{u_j} \in \R^{p}, \pmb{v_j} \in \R^{n-1}, \pmb{w_j} \in \R^{n-1}$, and $q_j \in \R$ for $j \in [n]$. 
The primary difference between the original Reformulation Linearization Technique (RLT) for binary mixed integer programs \cite{MR1061981} and ERLT is that in RLT we multiply the linear constraints of the problem with binary variables (and their complements) from the existing problem, while in ERLT, auxiliary variables are created, and the original linear constraints are multiplied with these auxiliary variables.  After multiplication, the reformulation and linearization steps are the same in RLT and ERLT.

Next, consider the vertex-cover-based relaxation for \eqref{LPCC} with a 1-regular conflict graph $G$.  When $G$ is 1-regular, the number of nodes of $G$ is even, and we assume without loss of generality that $G=([2n], E)$, where $E=\{\{i,i+n\} \mid i \in [n]\}$. 
Throughout this section, we define
$$
P^\perp_{1r} := \{(\textbf{\textit{x}}, \textbf{\textit{y}}) \in \R^{p+2n} \mid A \pmb{x} + D \pmb{y} \leq \pmb{d}, \ 0 \leq \pmb{y} \leq \textbf{1}_{2n}, \ y_i \cdot y_{n+i} = 0 \ \forall i \in [n]\}.
$$

Any minimum vertex cover $C$ of $G$ has the property that $|C \cap \{i, n+i\}| = 1$ for any $i \in [n]$, and by symmetry of variables $y_i$ and $y_{n+i}$ we may assume without loss of generality that the feasible cover partition is $\T=\{\{1\}, \ldots, \{n\}\}$.
Applying the vertex-cover-based extended relaxation procedure and renaming variables, we obtain the following extended relaxation for $P^\perp_{1r}$:
\begin{align*}
A \pmb{u_i} + D \pmb{\nu_i}  \leq \pmb{d}  q_i, \quad \forall i \in [n]  \\ 
A \pmb{\bar u_i} + D \pmb{\bar \nu_i}  \leq \pmb{d}  \bar{q}_i, \quad \forall i \in [n]  \\
\bar{\nu}_{i,i} = 0, \ \nu_{i,i+n} = 0, \quad \forall i \in [n] \\
0 \leq \pmb{\nu_i} \leq q_i \cdot \textbf{1}_{2n},  \ 0 \leq \pmb{\bar \nu_i} \leq \bar q_i \cdot \textbf{1}_{2n}, \quad \forall i \in [n]   \\
\pmb{\bar u_i} + \pmb{u_i} = \pmb{x}, \ \pmb{\bar \nu_i} + \pmb{\nu_i} = \pmb{y}, \quad \forall i \in [n]   \\
\bar{q}_i + q_i = 1, \ q_i, \bar{q}_i \geq 0, \quad \forall i \in [n].
\end{align*}
Here $\pmb{u_i}, \pmb{\bar u_i} \in \R^n, \pmb{\nu_i}, \pmb{\bar \nu_i} \in \R^{2n}$ for each $i \in [n]$. 
Now we denote $z_i: = y_{i+n}, v_{i,j} := \nu_{i,j}, \bar v_{i,j} := \bar \nu_{i,j}, w_{i,j} := \nu_{i,j+n}$, and $\bar w_{i,j} := \bar \nu_{i,j+n}$ for any $i,j \in [n]$, 
and let matrix $D := (B \mid C)$, where $B$ denotes the sub-matrix of $D$ given by the first $n$ columns, and $C$ denotes the sub-matrix of $G$ given by the remaining $n$ columns. Then
we have
\begin{align}
A \pmb{u_i} + B \pmb{v_i} + C \pmb{w_i}  \leq \pmb{d}  q_i, \quad \forall i \in [n]  \label{eq: inter1}\\ 
A \pmb{\bar u_i} + B \pmb{\bar v_i} + C \pmb{\bar w_i}  \leq \pmb{d}  \bar q_i, \quad \forall i \in [n]  \label{eq: inter2}\\ 
\bar{v}_{i,i} = 0, \ w_{i,i} = 0, \quad \forall i \in [n] \label{eq: inter3} \\
0 \leq \pmb{v_i}, \pmb{w_i} \leq q_i \cdot \textbf{1}_n,  \ 0 \leq \pmb{\bar v_i}, \pmb{\bar w_i} \leq \bar q_i \cdot \textbf{1}_n, \quad \forall i \in [n]  \label{eq: inter4} \\
\pmb{\bar u_i} + \pmb{u_i} = \pmb{x}, \ \pmb{\bar v_i} + \pmb{v_i} = \pmb{y_{[n]}}, \ \pmb{\bar w_i} + \pmb{w_i} = \pmb{z}, \quad \forall i \in [n]  \label{eq: inter5} \\
\bar{q}_i + q_i = 1, \ q_i, \bar{q}_i \geq 0, \quad \forall i \in [n]. \label{eq: inter6}
\end{align}
Here $\pmb{y_{[n]}} = (y_1, \ldots, y_n)^\transp$, and $\pmb{v_i}, \pmb{\bar v_i}, \pmb{w_i}, \pmb{\bar w_i} \in [0,1]^n$ for each $i \in [n]$. Lastly, for any $i \in [n]$, by eliminating variables $\bar{v}_{i,i}, w_{i,i}, \pmb{\bar u_i}, \pmb{\bar v_i}, \pmb{\bar w_i}$ and $\bar q_i$ using equations \eqref{eq: inter3},  \eqref{eq: inter5} and  \eqref{eq: inter6}, the above linear system will coincide exactly with the ERLT relaxation~\eqref{eq: erlt_formulation}, demonstrating their equivalence.

By this observation, the ERLT proposed in \cite{MR4270189} can be seen as a special case of our vertex-cover-based extended relaxation.  More specifically, by the construction of $\E_\T(P)$, we provide a disjunctive perspective for the ERLT procedure.  Once consequence of this different perspective is that Theorem~2 in \cite{MR4270189} can be deduced from the fact that the disjunctive formulation provides an extended formulation for the convex hull of multiple polyhedra with identical recession cones.
%, as we have proven in Section~\ref{theo: balas_disjun_nice}.

% In the next section, we investigate cutting-planes for the extended relaxation $\E_\T(P)$. Since ERLT \cite{MR4270189} is a special case of our formulation, as we will see in numerical experiments in Section~\ref{sec: experiments}, our cutting-planes can also be used to further strengthen the ERLT formulation for the LPCCs with 1-regular conflict graph.

%\begin{align*}
%A u_{i,\cdot} + B v_{i,\cdot}  \leq d \cdot q_i, & \quad \forall \{i,j\} \in E \\
%A \bar{u}_{i,\cdot} + B \bar{v}_{i,\cdot}  \leq d \cdot \bar{q}_i, & \quad \forall \{i,j\} \in E \\
%\bar{v}_{i,i} = 0, \ v_{i,j} = 0, & \quad \forall \{i,j\} \in E \\
%0 \leq v_{i, \cdot} \leq q_i \cdot e,  \ 0 \leq \bar v_{i, \cdot} \leq \bar q_i \cdot e, & \quad \forall \{i,j\} \in E \\
%\bar{u}_{i, \cdot} + u_{i, \cdot} = x, \ \bar{v}_{i, \cdot} + v_{i, \cdot} = y, & \quad \forall \{i,j\} \in E \\
%\bar{q}_i + q_i = 1, \ q_i, \bar{q}_i \geq 0, & \quad \forall \{i,j\} \in E.
%\end{align*}

\section{Cutting-planes in extended space}
\label{sec: cutting_plane}

In this section, we study how to strengthen the extended relaxation $\E_\T(P)$ through cutting planes.
First, we strengthen the relaxation by imposing additional valid bilinear equations:
\begin{align}
 \pmb{u_T} = q_T \cdot \pmb{x}, \quad \pmb{\bar u_T} = \bar q_T \cdot \pmb{x}  & \quad \forall T \in \T \label{eq: bilinear1} \\
 \pmb{v_T} = q_T \cdot \pmb{y}, \quad \pmb{\bar v_T} = \bar q_T \cdot \pmb{y} & \quad  \forall T \in \T, \label{eq: bilinear2} 
\end{align}
and we define the set
\begin{equation}
\label{eq: exact_curvy}
\tilde \E_{\T}(P) := \big\{(\pmb{x}, \pmb{y}, \pmb{q}, \pmb{\bar q}, \pmb{u}, \pmb{\bar u}, \pmb{v}, \pmb{\bar v}) \mid   \eqref{eq: balas 1}-\eqref{eq: balas 6}, \eqref{eq: bilinear1},\eqref{eq: bilinear2} \ \forall T \in \T \big\}.
\end{equation}

We first show that these bilinear equations are indeed valid by demonstrating that $\tilde \E_{\T}(P)$ is an extended formulation for $P^\perp$. 

\begin{proposition}
\label{prop: exact_formulation_curvy}
$\proj_{\textbf{\textit{x}},\textbf{\textit{y}}} \tilde \E_{\T}(P) = P^\perp.$
\end{proposition}

\begin{proof}
First, we show $\proj_{\textbf{\textit{x}},\textbf{\textit{y}}} \tilde \E_{\T}(P) \subseteq P^\perp.$ 
Pick $(\pmb{x^*}, \pmb{y^*}, \pmb{q^*}, \pmb{\bar q^*}, \pmb{u^*}, \pmb{\bar u^*}, \pmb{v^*}, \pmb{\bar v^*}) \in \tilde \E_{\T}(P)$. 
From constraints \eqref{eq: balas 1}, \eqref{eq: balas 2}, \eqref{eq: balas 4}, \eqref{eq: balas 5}, \eqref{eq: balas 6}, we know $(\pmb{x^*}, \pmb{y^*}) \in P$. 
Moreover, for each $T \in \T$, from \eqref{eq: balas 3}, \eqref{eq: bilinear1}, \eqref{eq: bilinear2}, we have $q^*_T \cdot y^*_j = 0$ for any $j \in \delta(T)$, and $\bar q^*_T \cdot y^*_i = 0$ for any $i \in T$. 
Since $\bar q^*_T + q^*_T = 1$ by \eqref{eq: balas 6}, we have either $\bar q^*_T > 0$ or $q^*_T > 0$. 
These yield that either $\pmb{y^*}(\delta(T)) = 0$ or $\pmb{y^*}(T) = 0$, which means $(\pmb{x^*}, \pmb{y^*}) \in P^\perp$ because of \eqref{eq: reform_Pperp}.

Now we arbitrarily pick $(\pmb{x'}, \pmb{y'}) \in P^\perp$. 
For each $T \in \T$, we define $q'_T = \mathbbm{1}\{\pmb{y'}(T) > 0\}$, $\bar q'_T = 1-q'_T$, and $\pmb{u'}, \pmb{\bar u'}, \pmb{v'}, \pmb{\bar v'}$ are defined by \eqref{eq: bilinear1} and \eqref{eq: bilinear2}, respectively. 
In order to show that $(\pmb{x'}, \pmb{y'}, \pmb{q'}, \pmb{\bar q'}, \pmb{u'}, \pmb{\bar u'}, \pmb{v'}, \pmb{\bar v'}) \in \tilde \E_{\T}(P)$, it suffices to show that \eqref{eq: balas 3} holds for each $T \in \T$, since all the remaining constraints are naturally satisfied from the construction of $\pmb{u'}, \pmb{\bar u'}, \pmb{v'}, \pmb{\bar v'}$. 
For each $T \in \T$ and $i \in T$, 
$$
\bar v'_{T,i} = \bar q'_T \cdot y'_i = (1-\mathbbm{1}\{\pmb{y'}(T) > 0\}) \cdot y'_i = \mathbbm{1}\{\pmb{y'}(T) = 0\} \cdot y'_i = 0.
$$
Next we verify that $v'_{T,j} = 0$ for any $j \in \delta(T)$.
Since $\pmb{y'}$ satisfies the complementarity constraints of $P^\perp$, by \eqref{eq: reform_Pperp} we know that either $\pmb{y'}(T) = 0$ or $\pmb{y'}(\delta(T)) = 0$. 
If $\pmb{y'}(T) = 0$, then $q'_T= \mathbbm{1}\{\pmb{y'}(T) > 0\} = 0$, which implies $v'_{T,j} =q'_T \cdot y'_j = 0$ for any $j \in \delta(T)$; 
If $\pmb{y'}(\delta(T)) = 0$, then $y'_j = 0$ for any $j \in \delta(T)$, so we have $v'_{T,j} = q'_T \cdot y'_j = 0$ for any $j \in \delta(T)$. Hence we have shown that $P^\perp \subseteq \proj_{\textbf{\textit{x}},\textbf{\textit{y}}} \tilde \E_{\T}(P)$, which completes the proof.
\end{proof}

From the disjunctive interpretation of the auxiliary variables in defining $\E_{\T}(P)$ and $\tilde \E_{\T}(P)$, we can further obtain the following proposition. Here for graph $G$, a \emph{clique} of $G$ is defined to be a subset of nodes such that every two distinct nodes in the clique are adjacent.

\begin{proposition}
\label{prop: simple_qcut}
Let $\T$ be a feasible cover partition of the conflict graph $G$, and let $C$ be a clique of $G$. 
If $T_1, \ldots, T_k$ are different node sets in $\T$ such that $T_i \cap C \neq \emptyset$ for any $i\in [k]$, then:
$$
\proj_{\textbf{\textit{x}}, \textbf{\textit{y}}} \left\{(\pmb{x}, \pmb{y}, \pmb{q}, \pmb{\bar q}, \pmb{u}, \pmb{\bar u}, \pmb{v}, \pmb{\bar v}) \in \tilde \E_{\T}(P) \mid  \sum_{i=1}^k q_{T_i} \leq 1, \sum_{i=1}^k \pmb{v_{T_i}} \leq \pmb{y} \right\} = P^\perp.
$$
\end{proposition}

\begin{proof}
From Section~\ref{prop: exact_formulation_curvy}, we only need to show the projection onto $(\textbf{\textit{x}}, \textbf{\textit{y}})$-space of the set defined in the statement of this proposition contains $P^\perp$.

Arbitrarily pick $(\pmb{x'}, \pmb{y'}) \in P^\perp$. 
For each $T \in \T$, we define $q'_T = \mathbbm{1}\{\pmb{y'}(T) > 0\}$, $\bar q'_T = 1-q'_T$, and $\pmb{u'}, \pmb{\bar u'}, \pmb{v'}, \pmb{\bar v'}$ are defined by \eqref{eq: bilinear1} and \eqref{eq: bilinear2}, respectively. 
It is obvious that $(\pmb{x'}, \pmb{y'}, \pmb{q'}, \pmb{\bar q'}, \pmb{u'}, \pmb{\bar u'}, \pmb{v'}, \pmb{\bar v'}) \in \tilde \E_{\T}(P)$.
Let $C$ and $T_1, \ldots, T_k \in \T$ be the node sets defined in the statement of this proposition. 

First of all, we show that $\sum_{i=1}^k q'_{T_i} \leq 1$.
If $\sum_{i=1}^k q'_{T_i} > 1$, from the definition of $q'_T = \mathbbm{1}\{\pmb{y'}(T) > 0\}$, we know there are $i_1 \neq i_2 \in [k]$ such that $\pmb{y'}(T_{i_1}) > 0, \pmb{y'}(T_{i_2}) > 0$. 
Say $y'_{j_1} > 0, y'_{j_2} > 0$ for $j_1 \in T_{i_1}, j_2 \in T_{i_2}$. 
Since $T_{i_1} \cap C \neq \emptyset, T_{i_2} \cap C \neq \emptyset$, and $C$ is a clique, we know there exists $j_3 \in T_{i_1}$ and $j_4 \in T_{i_2}$, such that $(j_3, j_4)$ is an edge in $G$. Hence $j_3 \in \delta(j_4), j_4 \in \delta(j_3)$. 
By definition of feasible cover partition, $\delta(T_{i_1}) = \delta(j_1) = \delta(j_3), \delta(T_{i_2}) = \delta(j_2) = \delta(j_4)$. 
Therefore, $j_3 \in \delta(j_2)$, so $j_2 \in \delta(j_3) = \delta(j_1)$, which means $(j_1, j_2)$ is also an edge of $G$. 
However, $y'_{j_1} > 0$, $y'_{j_2} > 0$, thus we get the contradiction using the complementarity constraints.

Lastly, $\sum_{i=1}^k \pmb{v'_{T_i}} \leq \pmb{y'}$ follows immediately from \eqref{eq: bilinear2} in the definition of $\tilde \E_{\T}(P)$. 
\end{proof}

In particular, Section~\ref{prop: simple_qcut} implies that all inequalities in Section~\ref{prop: simple_qcut} can be added to the formulation of $\E_{\T}(P)$ to obtain a stronger extended relaxation.

Even though Section~\ref{prop: exact_formulation_curvy} shows that $\tilde \E_{\T}(P)$ is a formulation of $P^\perp$ in an extended space, in general $\conv(\tilde \E_{\T}(P))$ is not expected to be polyhedral.
In order to obtain a polyhedral relaxation of $\conv(\tilde \E_{\T}(P))$ which is stronger than the previous extended relaxation $ \E_{\T}(P)$, we define $\S_{\T}$ as the set of points $(\pmb{y}, \pmb{q}, \pmb{\bar q}, \pmb{v}, \pmb{\bar v})$ in $\R^{n+2|\T|+2|\T|n}$ that satisfy the constraints defining $\tilde \E_{\T}(P)$ in \eqref{eq: exact_curvy} which do not involve variables $x, u$ and $\bar u$. 
Namely
\begin{equation}
\label{eq: defn_curvyS}
\S_{\T} :=  \big\{(\pmb{y}, \pmb{q}, \pmb{\bar q}, \pmb{v}, \pmb{\bar v}) \mid  \eqref{eq: balas 3}, \eqref{eq: balas 4}, \eqref{eq: balas 6}, \eqref{eq: bilinear2} \ \forall T \in \T \big\}.
\end{equation}
Here we ignore the equation $\pmb{\bar v_T} + \pmb{v_T} = \pmb{y}$ in \eqref{eq: balas 5} since it is induced by \eqref{eq: bilinear2} and $q_T + \bar q_T = 1$.
Note that:
\begin{equation}
\label{eq: bqp_relaxation}
\begin{split}
\conv(\tilde \E_{\T}(P)) \subseteq \big\{ (\pmb{x},\pmb{y},\pmb{q}, \pmb{\bar q}, \pmb{u},\pmb{\bar u}, \pmb{v}, \pmb{\bar v}) \mid \ &  \pmb{x} = \pmb{u_T} + \pmb{\bar u_T}, \eqref{eq: balas 1}, \eqref{eq: balas 2},\\
& (\pmb{y}, \pmb{q}, \pmb{\bar q}, \pmb{v}, \pmb{\bar v}) \in \conv(\S_{\T}) 
 \ \forall T \in \T \big\}.
\end{split}
\end{equation}

Obviously, the relaxation of $\conv(\tilde \E_{\T}(P))$ in \eqref{eq: bqp_relaxation} is tighter than $\E_{\T}(P)$. 
In the next result we show that it is indeed a polyhedral relaxation.

\begin{proposition}
\label{prop: binary_extreme}
The extreme points of $\S_{\T}$ are contained in $\{0,1\}^{n+2|\T|+2|\T|n}$. Hence the set $\conv(\S_{\T})$ is polyhedral.
\end{proposition}

\begin{proof}
Let $\pmb{p^*} = (\pmb{y^*}, \pmb{q^*}, \pmb{\bar q^*}, \pmb{v^*}, \pmb{\bar v^*})$ be an extreme point of $\S_{\T}$. 
It suffices to show that $\pmb{y^*} \in \{0,1\}^n$ and $\pmb{q^*} \in \{0,1\}^{|\T|}$, since from equations $\bar q^*_T = 1- q^*_T, \pmb{v^*_T} = q^*_T \cdot \pmb{y^*}$ and $\pmb{\bar v^*_T} = \bar q^*_T \cdot \pmb{y^*}$, the integrality of the remaining variables follows.

First, we show $\pmb{y^*} \in \{0,1\}^n$. 
Assume for a contradiction that $0 < y^*_\ell < 1$ for some $\ell \in [n]$. Let $\pmb{p^0} = (\pmb{y^0}, \pmb{q^0}, \pmb{\bar q^0}, \pmb{v^0}, \pmb{\bar v^0})$ be the vector in $\R^{n+2|\T| +  2|\T| n}$ obtained from $\pmb{p^*}$ by setting $y^0_\ell = 0$ and $v^0_{T, \ell} = \bar v^0_{T, \ell} = 0$ for any $T \in \T$. Similarly, let $\pmb{p^1}$ be obtained from $\pmb{p^*}$ by setting $y^1_\ell = 1$ and $v^1_{T,\ell} = q^*_T, \bar v^1_{T, \ell} = \bar q^*_T$ for any $T \in \T$. 

Now we show that both vectors $\pmb{p^0}$ and $\pmb{p^1}$ are in $\S_{\T}$. From the construction of $\pmb{p^0}$ and $\pmb{p^1}$, it suffices to check that \eqref{eq: balas 3} are satisfied by both $\pmb{p^0}$ and $\pmb{p^1}$. Note that $\pmb{p^0} \leq \pmb{p^*}$ (component-wise), so the fact that $\pmb{p^*}$ satisfies \eqref{eq: balas 3} directly implies that $\pmb{p^0}$ satisfies \eqref{eq: balas 3}. 
For $\pmb{p^1}$, we want to check that: for any $T \in \T$, $\bar v^1_{T, i} = 0 \ \forall i \in T$, and $v^1_{T, j} = 0 \ \forall j \in \delta(T)$. 
Since $\pmb{p^*}$ satisfies \eqref{eq: balas 3}, and from the construction of $\pmb{p^1}$, it suffices to check that: $q^*_T = v^1_{T, \ell} = 0$ for any $T \in \T$ with $T \cap \delta(\ell) \neq \emptyset$, and $\bar q^*_T = \bar v^1_{T, \ell} = 0$ for any $T \in \T$ with $\ell \in T$. 
Since $\pmb{p^*}$ satisfies \eqref{eq: balas 3} and $\pmb{v^*_{T}} = q^*_T \cdot \pmb{y^*}, \pmb{\bar v^*_{T}} = \bar q^*_T \cdot \pmb{y^*}$ for any $T \in \T$, we obtain that $\bar v^*_{T, \ell} = \bar q^*_T \cdot y^*_\ell = 0$ for any $T \in \T$ with $\ell \in T$, and $v^*_{T, \ell} = q^*_T \cdot y^*_\ell = 0$ for any $T \in \T$ with $T \cap \delta(\ell) \neq \emptyset$. By assumption of $y^*_\ell > 0$, we get $\bar q^*_T = 0$ for any $T \in \T$ with $\ell \in T$, and $q^*_T = 0$ for any $T \in \T$ with $T \cap \delta(\ell) \neq \emptyset$. Therefore, we have shown that both $\pmb{p^0}, \pmb{p^1}$ are in $\S_{\T}$.

Next, we show that $\pmb{p^*}$ can be written as the convex combination of $\pmb{p^0}$ and $\pmb{p^1}$:
\begin{equation}
\label{eq: cc_check}
\pmb{p^*} = y^*_\ell \cdot \pmb{p^1} + (1-y^*_\ell) \cdot \pmb{p^0}.
\end{equation}
From the construction of $\pmb{p^0}$ and $\pmb{p^1}$, we only need to consider the components of $y_\ell$ and $v_{T, \ell}, \bar v_{T, \ell}$ for all $T \in \T$. The $y_\ell$ component of \eqref{eq: cc_check} is trivial to check. For any $T \in \T$, we have
\begin{align*}
& v^*_{T,\ell} = y^*_\ell \cdot q^*_T + (1-y^*_\ell) \cdot 0 = y^*_\ell \cdot v^1_{T, \ell} + (1-y^*_\ell) \cdot v^0_{T, \ell},\\
& \bar v^*_{T,\ell} = y^*_\ell \cdot \bar q^*_T + (1-y^*_\ell) \cdot 0 = y^*_\ell \cdot \bar v^1_{T, \ell} + (1-y^*_\ell) \cdot \bar v^0_{T, \ell},
\end{align*}
where the first equations of both lines are from the fact that $\pmb{p^*} \in \S_{\T}$, and the second equations are from the definition of $\pmb{p^0}$ and $\pmb{p^1}$. 
We have thereby shown that, if $\pmb{p^*}$ has fractional $y$ components, then $\pmb{p^*}$ can be written as the convex combination of two different points in $\S_{\T}$. 
This contradicts the fact that $\pmb{p^*}$ is an extreme point.

Next, we show $\pmb{q^*} \in \{0,1\}^{|\T|}$. 
Assume for a contradiction that $0 < q^*_{\tilde T} < 1$ for some $\tilde T \in \T$. 
Let $\pmb{p^2}=(\pmb{y^2}, \pmb{q^2}, \pmb{\bar q^2}, \pmb{v^2}, \pmb{\bar v^2})$ be the vector in $\R^{n+2|\T| + 2|\T| n}$ obtained from $\pmb{p^*}$ by setting $q^2_{\tilde T} = 0$, $\bar q^2_{\tilde T} = 1$, $\pmb{v^2_{\tilde T}} = 0$, $\pmb{\bar v^2_{\tilde T}} = \pmb{y^*}$. 
Similarly, let $\pmb{p^3}$ be obtained from $\pmb{p^*}$ by setting $q^3_{\tilde T} = 1$, $\bar q^3_{\tilde T} = 0$, $\pmb{v^3_{\tilde T}} = \pmb{y^*}$, $\pmb{\bar v^3_{\tilde T}} = 0$. 
Now, we show that $\pmb{p^2}, \pmb{p^3} \in \S_{\T}$. 
Here we only have to show that \eqref{eq: balas 3} are satisfied by both $\pmb{p^2}$ and $\pmb{p^3}$. 
For $\pmb{p^2}$, we need to show: $\bar v^2_{\tilde T, i} = 0 \ \forall i \in \tilde T$ and $v^2_{\tilde T, j} = 0$ for any $j \in \delta(T)$. 
By definition of $p^2$, we only need to show $y^*_i = 0 \ \forall i \in \tilde T$. 
By assumption of $q^*_{\tilde T} \in (0,1)$, we also have $\bar q^*_{\tilde T} \in (0,1)$. 
Since \eqref{eq: balas 3} and \eqref{eq: bilinear2} imply $\bar v^*_{\tilde T, i} = \bar q^*_{\tilde T} \cdot y^*_i = 0$, for any $i \in \tilde T$ we have $y^*_i = 0$ for any $i \in \tilde T$. Hence we have shown $\pmb{p^2} \in \S_{\T}$. Similarly, we can obtain $\pmb{p^3} \in \S_{\T}$. 
Lastly, we show that:
\begin{equation}
\label{eq: cc_check2}
\pmb{p^*} = q^*_{\tilde T} \cdot \pmb{p^3} + (1-q^*_{\tilde T}) \cdot \pmb{p^2}.
\end{equation}
From the definition of $\pmb{p^2}$ and $\pmb{p^3}$, we only need to consider the components of $q_{\tilde T}, \bar q_{\tilde T}, v_{\tilde T, j}, \bar v_{\tilde T, j}$ for any $j \in [n]$. The $q_{\tilde T}, \bar q_{\tilde T}$ components of \eqref{eq: cc_check2} are trivial to check. For any $j \in [n]$, we have
\begin{align*}
& v^*_{\tilde T, j} = q^*_{\tilde T} \cdot y^*_j + (1-q^*_{\tilde T}) \cdot 0 = q^*_{\tilde T} \cdot v^3_{\tilde T,j} + (1-q^*_{\tilde T}) \cdot v^2_{\tilde T,j}, \\
& \bar v^*_{\tilde T, j} = q^*_{\tilde T} \cdot 0 + \bar q^*_{\tilde T} \cdot y^*_j = q^*_{\tilde T} \cdot \bar v^3_{\tilde T,j} + (1-q^*_{\tilde T}) \cdot \bar v^2_{\tilde T,j},
\end{align*}
where the first equations of both lines are from the fact that $\pmb{p^*} \in \S_{\T}$, and the second equations are from the definition of $\pmb{p^2}$ and $\pmb{p^3}$. 
Therefore, $\pmb{p^*}$ can be written as the convex combination of two different points in $\S_{\T}$, which contradicts the fact that $\pmb{p^*}$ is an extreme point.
\end{proof}

According to Section~\ref{prop: binary_extreme}, the convex hull of $\S_{\T}$ in \eqref{eq: defn_curvyS} can be rewritten as:
\begin{equation}
\label{eq: cvx_rewrite_S}
\conv(\big\{ (\pmb{y}, \pmb{q}, \pmb{\bar q}, \pmb{v}, \pmb{\bar v}) \in \{0,1\}^{n+2|\T|+2|\T|n} \mid   \eqref{eq: balas 3}, \eqref{eq: balas 6}, \eqref{eq: bilinear2} \ \forall T \in \T \big\}).
\end{equation}
Here constraints \eqref{eq: balas 4} are redundant hence they have been removed. 

For the purpose of studying quadratic 0-1 programming problems, Padberg \cite{MR1017216} proposed
the \emph{boolean quadric polytope} (BQP), which is a polytope in a higher dimensional space that results from the linearization of the quadratic terms. 
%This concept was originally defined with respect to a complete graph, but it has natural generalization to sparser graph. 
In particular, the BQP with respect to the complete bipartite graph 
that can be bipartitioned into $2|\T|$ and $n$ nodes is defined as
$$
\QP_{2|\T|, n} := \conv(\{ (\pmb{y}, \pmb{q}, \pmb{\bar q}, \pmb{v}, \pmb{\bar v}) \in \{0,1\}^{n+2|\T|+2|\T|n} \mid  \eqref{eq: bilinear2} \ \forall T \in \T \}).
$$
In fact, this special case of BQP corresponding to a complete bipartite graph is also called \emph{bipartite BQP}, and its structural properties have been recently studied in \cite{sripratak2021bipartite}.
Using this definition, by \eqref{eq: cvx_rewrite_S}, we have the following relaxation for $\conv(\S_{\T}):$
\begin{equation}
\label{eq: bqp_strong_relaxation}
\begin{split}
 \conv(\S_{\T}) \subseteq 
& \conv(\{ (\pmb{y}, \pmb{q}, \pmb{\bar q}, \pmb{v}, \pmb{\bar v}) \in \{0,1\}^{n+2|\T|+2|\T|n} \mid  \eqref{eq: bilinear2} \ \forall T \in \T\} )  \\
& \cap \{(\pmb{y}, \pmb{q}, \pmb{\bar q}, \pmb{v}, \pmb{\bar v})  \mid  \eqref{eq: balas 3}, \eqref{eq: balas 6} \ \forall T \in \T, \pmb{\bar{v}_{T}} + \pmb{v_{T}} = \pmb{y}\} \\
= & \{(\pmb{y}, \pmb{q}, \pmb{\bar q}, \pmb{v}, \pmb{\bar v}) \in \QP_{2|\T|, n}  \mid  \eqref{eq: balas 3}, \eqref{eq: balas 6} \ \forall T \in \T, \pmb{\bar{v}_{T}} + \pmb{v_{T}} = \pmb{y}\}.
\end{split}
\end{equation}
Here we also include the equations $\pmb{\bar{v}_{T}} + \pmb{v_{T}} = \pmb{y}$ for all $T \in \T$, since they are obviously valid for $\S_{\T}$.
Interestingly, the next theorem shows that, the relaxation given in \eqref{eq: bqp_strong_relaxation} is in fact a formulation of $\conv(\S_{\T})$. 
It is worth mentioning that such result is non-trivial in general, since \eqref{eq: balas 6} and $\pmb{\bar{v}_{T}} + \pmb{v_{T}} = \pmb{y}$ are not facial constraints to $\QP_{2|\T|, n}$, and equations $\pmb{\bar{v}_{T}} + \pmb{v_{T}} = \pmb{y}$ are redundant in the definition \eqref{eq: defn_curvyS} of $\S_{\T}$, while they are necessary in order for the statement of the next theorem to hold.

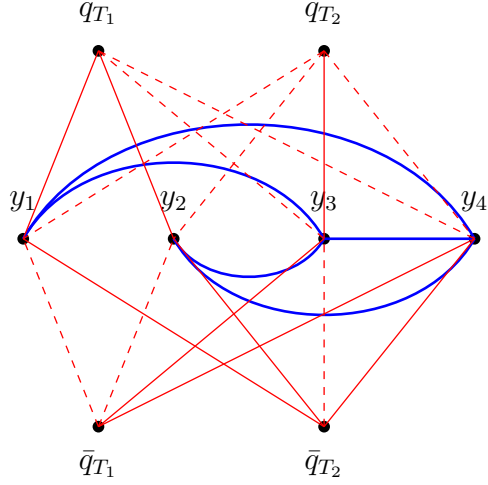
\begin{figure}[h]
\centering
\begin{tikzpicture}
%     \draw (3.11, 1.2) node[fill,circle,scale=0.4] {};
%     \draw (3.11, -1.2) node[fill,circle,scale=0.4] {};
%     \draw (7.11,5.2) node[fill,circle,scale=0.4] {};
%     \draw (9.51,5.2) node[fill,circle,scale=0.4] {};    
%     \draw (13.51,1.2) node[fill,circle,scale=0.4] {};    
%     \draw (13.51,-1.2) node[fill,circle,scale=0.4] {};    
%     \draw (7.11,-5.2) node[fill,circle,scale=0.4] {};    
%     \draw (9.51,-5.2) node[fill,circle,scale=0.4] {};     
%     \draw (1.91,1.2) node[] {$y_1$};
%     \draw (1.91,-1.2) node[] {$y_2$};
%     \draw (14.71,1.2) node[] {$z_1$};
%     \draw (14.71,-1.2) node[] {$z_2$};
%     \draw (7.11,-6.4) node[] {$\bar q_1$};
%     \draw (9.51,-6.4) node[] {$\bar q_2$};
%     \draw (7.11,6.4) node[] {$q_1$};
%     \draw (9.51,6.4) node[] {$q_2$};
\draw (0,-0.5) node[fill,circle,scale=0.4] {};     
\draw (0,0) node[] {$y_1$};
\draw (2,-0.5) node[fill,circle,scale=0.4] {};     
\draw (2,0) node[] {$y_2$};
\draw (4,-0.5) node[fill,circle,scale=0.4] {};     
\draw (4,0) node[] {$y_3$};
\draw (6,-0.5) node[fill,circle,scale=0.4] {};     
\draw (6,0) node[] {$y_4$};
\draw (1,2) node[fill,circle,scale=0.4] {};     
\draw (1,2.5) node[] {$q_{T_1}$};
\draw (4,2) node[fill,circle,scale=0.4] {};     
\draw (4,2.5) node[] {$q_{T_2}$};
\draw (1,-3) node[fill,circle,scale=0.4] {};     
\draw (1,-3.5) node[] {$\bar q_{T_1}$};
\draw (4,-3) node[fill,circle,scale=0.4] {};     
\draw (4,-3.5) node[] {$\bar q_{T_2}$};

%\draw (5.5,1.1) node[] {$ v_{T, i}$};
%\draw (0,-1.8) node[] {$\bar v_{T, i}$};

\draw [blue, line width = 1](0,-0.5) to[out=60,in=120] (4,-0.5);
\draw [blue, line width = 1](0,-0.5) to[out=60,in=120] (6,-0.5);
\draw [blue, line width = 1](2,-0.5) to[out=-60,in=-120] (4,-0.5);
\draw [blue, line width = 1](2,-0.5) to[out=-60,in=-120] (6,-0.5);
\draw[blue, line width = 1](4,-0.5) -- (6,-0.5);

\draw[red, dashed, line width = 0.5,-] (1,-3) -- (0,-0.5);
\draw[red, dashed, line width = 0.5,-] (1,-3) -- (2,-0.5);
\draw[red, dashed, line width = 0.5,-] (4,-3) -- (4,-0.5);
\draw[red, line width = 0.5,-] (1,2) -- (0,-0.5);
\draw[red, line width = 0.5,-] (1,2) -- (2,-0.5);
\draw[red, line width = 0.5,-] (4,2) -- (4,-0.5);
\draw[red, dashed, line width = 0.5,-] (1,2) -- (4,-0.5);
\draw[red, dashed, line width = 0.5,-] (1,2) -- (6,-0.5);
\draw[red, dashed, line width = 0.5,-] (4,2) -- (0,-0.5);
\draw[red, dashed, line width = 0.5,-] (4,2) -- (2,-0.5);
\draw[red, dashed, line width = 0.5,-] (4,2) -- (6,-0.5);
\draw[red, line width = 0.5,-] (1,-3) -- (4,-0.5);
\draw[red, line width = 0.5,-] (1,-3) -- (6,-0.5);
\draw[red, line width = 0.5,-] (4,-3) -- (0,-0.5);
\draw[red, line width = 0.5,-] (4,-3) -- (2,-0.5);
\draw[red, line width = 0.5,-] (4,-3) -- (6,-0.5);
%\draw[line width = 1,-] (3.11,1.2) -- (7.11,5.2);
%\draw[line width = 1,-] (3.11,-1.2) -- (9.51,5.2);
%\draw[line width = 1,-] (3.11,1.2) -- (9.51,5.2);
%\draw[line width = 1,-] (3.11,-1.2) -- (7.11,5.2);
%\draw[line width = 1,-] (13.51,1.2) -- (9.51,5.2);
%\draw[line width = 1,-] (13.51,-1.2) -- (7.11,5.2);
%
%\draw[dashed, line width = 1,-] (13.51,1.2) -- (7.11,5.2);
%\draw[dashed, line width = 1,-] (13.51,-1.2) -- (9.51,5.2);
%\draw[dashed, line width = 1,-] (3.11,-1.2) -- (9.51,-5.2);
%\draw[dashed, line width = 1,-] (3.11,1.2) -- (7.11,-5.2);
%
%\draw[line width = 1,-] (3.11,-1.2) -- (7.11,-5.2);
%\draw[line width = 1,-] (3.11,1.2) -- (9.51,-5.2);
%\draw[line width = 1,-] (7.11,-5.2) -- (13.51,1.2);
%\draw[line width = 1,-] (9.51,-5.2) -- (13.51,-1.2);
%\draw[line width = 1,-] (7.11,-5.2) -- (13.51,-1.2);
%\draw[line width = 1,-] (9.51,-5.2) -- (13.51,1.2);
%
%     \draw (4.31,4.0) node[] {$v^i_j$};
%     \draw (8.31,7.2) node[] {index $i$};
%     \draw (8.31,-7.2) node[] {index $i$};
%     \draw (15.51,0.0) node[] {index $j$};
%     \draw (1.11,0.0) node[] {index $j$};
%     \draw (12.31,4.0)  node[] {$w^i_j$};
%     \draw (4.31,-4.0) node[] {$\bar v^i_j$};
%     \draw (12.31,-4.0) node[] {$\bar w^i_j$};
 \end{tikzpicture}
\caption{
In this figure, the original conflict graph $G$ is marked in blue, and the complete bipartite graph is marked in red. $\T= \{T_1, T_2\} := \left\{\{1, 2\}, \{3\} \right\}$ is a feasible cover partition of $G$, each red edge represents the variable $v_{T,i}$ or $\bar v_{T,i}$ for $T \in \T$ and $i \in [4]$, and the dashed red edges therein represent the variables that are set to 0 in \eqref{eq: balas 3}, which are $\bar v_{T_1, 1}, \bar v_{T_1, 2}, \bar v_{T_2, 3}, v_{T_1, 3}, v_{T_1, 4}, v_{T_2, 1}, v_{T_2, 2}$ and $v_{T_2, 4}$. 
}
\label{fig: G}
\end{figure}

\begin{theorem}
\label{th above}
$\{(\pmb{y}, \pmb{q}, \pmb{\bar q}, \pmb{v}, \pmb{\bar v}) \in \QP_{2|\T|, n}  \mid  \eqref{eq: balas 3}, \eqref{eq: balas 6} \ \forall T \in \T, \pmb{\bar{v}_{T}} + \pmb{v_{T}} = \pmb{y}\} = \conv(\S_{\T})$. 
\end{theorem}

\begin{proof}
%First of all, note that \eqref{eq: balas 3} are facial constraints to $\QP_{2c, n}$, so there is
%$\{(y, q, \bar q, v, \bar v) \in \QP_{2c, n}  \mid \forall T \in \T: \eqref{eq: balas 3}, \eqref{eq: balas 6}\}$
Because of \eqref{eq: bqp_strong_relaxation}, here we only have to show the other direction:
$  \{(\pmb{y}, \pmb{q}, \pmb{\bar q}, \pmb{v}, \pmb{\bar v}) \in \QP_{2|\T|, n}  \mid  \eqref{eq: balas 3}, \eqref{eq: balas 6} \ \forall T \in \T, \pmb{\bar{v}_{T}} + \pmb{v_{T}} = \pmb{y}\} \subseteq \conv(\S_{\T})$. 

Arbitrarily pick $p^* = (\pmb{y^*}, \pmb{q^*}, \pmb{\bar q^*}, \pmb{v^*}, \pmb{\bar v^*}) \in  \QP_{2|\T|, n}$ which also satisfies \eqref{eq: balas 3}, \eqref{eq: balas 6}, and $\pmb{\bar{v}_{T}} + \pmb{v_{T}} = \pmb{y}$ for every $T \in \T$. 
We want to show that $\pmb{p^*} \in \conv(\S_{\T})$.
Here $\pmb{p^*}$ can be written as the convex combination of some binary points in $\QP_{2|\T|, n}$: 
\begin{equation}
\label{eq: cc_bqp}
\pmb{p^*} = \sum_{i=1}^k \lambda_i \pmb{p^i}, \ \lambda_i > 0, \ \sum_{i=1}^k \lambda_i = 1, \ \pmb{p^i} \in \QP_{2|\T|, n} \cap \{0,1\}^{n+2|\T|+2|\T|n}.
\end{equation}
Our goal is to show that either all of those $\pmb{p^i}$ are contained in $\S_{\T}$, or there exist some other binary points in $\S_{\T}$ which contain $\pmb{p^*}$ in their convex hull.

First of all, since $\pmb{p^*}$ satisfies \eqref{eq: balas 3} for any $T \in \T$, and $\pmb{p^*}$ can be written as the convex combination in \eqref{eq: cc_bqp}, we know that $\pmb{p^i}$ satisfies \eqref{eq: balas 3} for all $i \in [k]$. If for any $i \in [k]$ and $T \in \T$, we also have $q^i_T + \bar q^i_T = 1$, then we complete the proof, since all of these $\pmb{p^i} \in \S_{\T}$, and $\pmb{p^*} = \sum_{i=1}^k \lambda_i \pmb{p^i} \in \conv(\S_{\T})$. 

Now, assume equation $q_{\tilde T} + \bar q_{\tilde T} = 1$ is not satisfied by all $\{\pmb{p^i}\}_{i=1}^k$, for some $\tilde T \in \T$. 
Since $q_{\tilde T} + \bar q_{\tilde T} = 1$ is satisfied by $\pmb{p^*}$, we know there must exist some point $\pmb{p^r}$ with $q^r_{\tilde T} + \bar q^r_{\tilde T} = 0$ and some other point $\pmb{p^s}$ with $q^s_{\tilde T} + \bar q^s_{\tilde T} = 2$. Namely, there exist $\{\ell_i\}_{i=1}^a, \{h_i\}_{i=1}^b \subseteq [k]$ such that $q^{\ell_i}_{\tilde T} + \bar q^{\ell_i}_{\tilde T} = 0$ for all $i \in [a]$, and $q^{h_i}_{\tilde T} + \bar q^{h_i}_{\tilde T} = 2$ for all $i \in [b]$. Next, we want to transform $\pmb{p^{\ell_1}}, \ldots, \pmb{p^{\ell_a}}, \pmb{p^{h_1}}, \ldots, \pmb{p^{h_b}}$ into some other binary points in $\QP_{2|\T|, n}$ with $q_{\tilde T} + \bar q_{\tilde T} = 1$. For all points in $\{\pmb{p^i}\}_{i=1}^k$ with index in $\{\ell_i\}_{i=1}^a \cup \{h_i\}_{i=1}^b$, we replace their $\bar q_{\tilde T}$ component by $1-\bar q_{\tilde T}$, and change those components $\pmb{\bar v_{\tilde T}}$ correspondingly to make $\pmb{p^i}$ remain feasible in $\QP_{2|\T|, n}$. 
We denote the obtained set of binary points by $\{\pmb{\tilde p^i}\}_{i=1}^k$. 
Clearly, after such transformation, those new points now all satisfy the equation $q_{\tilde T} + \bar q_{\tilde T} = 1$. Now we verify that $\pmb{p^*}$ can still be written as their convex combination, with the same convex combination coefficients: $\pmb{p^*} = \sum_{i=1}^k \lambda_i \pmb{\tilde p^i}$ . Note that the only components changed in $\sum_{i=1}^k \lambda_i \pmb{\tilde {p^i}}$ from $\sum_{i=1}^k \lambda_i \pmb{p^i}$ are $\bar q_{\tilde T}$ and $\pmb{\bar v_{\tilde T}}$. From the transformation above, the change occurred to component $\bar q_{\tilde T}$ of $\sum_{i=1}^k \lambda_i \pmb{\tilde p^i}$ from $\sum_{i=1}^k \lambda_i \pmb{p^i}$ is: $\sum_{i=1}^a \lambda_{\ell_i} - \sum_{i=1}^b \lambda_{h_i}$. By the fact that $q^*_{\tilde T} + \bar q^*_{\tilde T} = 1$ and $\pmb{p^*} = \sum_{i=1}^k \lambda_i \pmb{p^i}$, we have: 
$$
1 = 1\cdot (1-\sum_{i=1}^a \lambda_{\ell_i} - \sum_{i=1}^b \lambda_{h_i}) + 0 \cdot \sum_{i=1}^a \lambda_{\ell_i} + 2 \cdot \sum_{i=1}^b \lambda_{h_i},
$$
which gives us $\sum_{i=1}^a \lambda_{\ell_i} - \sum_{i=1}^b \lambda_{h_i} = 0$. So the component $\bar q_{\tilde T}$ of $\sum_{i=1}^k \lambda_i \pmb{\tilde p^i}$ remains unchanged from $\sum_{i=1}^k \lambda_i \pmb{p^i}$. Similarly, the change occurred to components $\pmb{\bar v_{\tilde T}}$ of $\sum_{i=1}^k \lambda_i \pmb{\tilde p^i}$ from $\sum_{i=1}^k \lambda_i \pmb{p^i}$ is: $\sum_{i=1}^a \lambda_{\ell_i} \pmb{y^{\ell_i}} - \sum_{i=1}^b \lambda_{h_i}\pmb{ y^{h_i}}$, where $\pmb{y^j}$ are the $\pmb{y}$ components of point $\pmb{p^j}$ for any $j \in [k]$. 
According to the fact that $\pmb{y^*} = \pmb{v^*_{\tilde T}} + \pmb{\bar v^*_{\tilde T}}$ and $\pmb{p^*} = \sum_{i=1}^k \lambda_i \pmb{p^i}$, here we have:
$$
\sum_{i=1}^a \lambda_{\ell_i} \pmb{y^{\ell_i}} + \sum_{i \neq \ell_j \text{ nor } h_j} \lambda_i \pmb{y^i} + \sum_{i=1}^b \lambda_{h_i} \pmb{y^{h_i}}  = \pmb{y^*} = \pmb{v^*_{\tilde T}} + \pmb{\bar v^*_{\tilde T}} = \sum_{i \neq \ell_j \text{ nor } h_j} \lambda_i \pmb{y^i} + 2 \sum_{i=1}^b \lambda_{h_i} \pmb{y^{h_i}}.
$$
This gives us $\sum_{i=1}^a \lambda_{\ell_i} \pmb{y^{\ell_i}} = \sum_{i=1}^b \lambda_{h_i} \pmb{y^{h_i}}$, which means that the components $\pmb{\bar v_{\tilde T}}$ of $\sum_{i=1}^k \lambda_i \pmb{\tilde p^i}$ remain unchanged from $\sum_{i=1}^k \lambda_i \pmb{p^i}$. 
Therefore, $\pmb{p^*} = \sum_{i=1}^k \lambda_i \pmb{p^i} = \sum_{i=1}^k \lambda_i \pmb{\tilde p^i}$. 
The above argument holds for any $T \in \T$ which has $q^i_T + \bar q^i_T \neq 1$ for some $i \in [k]$. 
Eventually, we end up getting $k$ binary points in $\S_{\T}$ that contain $\pmb{p^*}$ in their convex hull, hence $\pmb{p^*} \in \conv(\S_{\T})$.
\end{proof}

From Section~\ref{th above}, in order to fully utilize the cutting-planes provided by $\conv(\S)$, one only needs to consider valid inequalities of $\QP_{2|\T|, n}$ as well as some additional equations in \eqref{eq: bqp_strong_relaxation}.
However, it is challenging to enumerate the facets of $\QP$ even for a very small graph, see \cite{MR3508431}. 
The polyhedral structure of $\QP$ for a series-parallel or an acyclic graph were studied originally by Padberg \cite{MR1017216}.
More recently, Sripratak et al. \cite{sripratak2021bipartite} studied valid inequalities for $\QP$ over a complete bipartite graph, which is particularly interesting for our setting and they coined it as \emph{bipartite boolean quadric polytope}. 
In that paper, the authors proposed a few families of valid inequalities, e.g., \emph{odd-cycle inequalities}, \emph{rounded clique inequalities}, \emph{rounded cut inequalities}, etc., among which \emph{odd-cycle inequalities} are the only family of inequalities that can be facet-defining. 
% One another paper \cite{barmann2020bipartite} further studied the bipartite boolean quadric polytope with multiple-choice constraints and analyzed its combinatorial properties. 
Therefore, even though we have a complete characterization of $\conv(\S)$ in terms of $\QP_{2|\T|, n}$, facet-defining inequalities of $\QP_{2|\T|, n}$ are still generally hard to obtain.
Now, we will generate valid inequalities of $\conv(\S)$ from a different perspective.

It is simple to see that, by studying the sub-system of \eqref{LPCC} involving only $y$ variables, we are also able to derive valid inequalities for $P^\perp$. Namely, any valid inequality for the convex set 
\begin{equation}
\label{eq: stable_set}
\conv(\{\pmb{y} \in [0,1]^n \mid y_i \cdot y_j = 0 \ \forall \{i,j\} \in E\})
\end{equation}
can be utilized as a cutting-plane involving only $y$-variables to strengthen the LP relaxation $P$. 
In fact, as we will see in the next result, the convex set \eqref{eq: stable_set} is simply the \emph{stable set polytope} of the graph $G$, and any valid inequality for such stable set polytope is actually valid for the projection of $\conv(\S_{\T})$.
 
\begin{proposition}
\label{prop: stableset}
For any graph $G$ and any feasible cover partition $\T$ of $G$, 
$\proj_{\pmb{y}} \conv(\S_{\T}) \subseteq \conv(\{\pmb{y} \in [0,1]^n \mid y_i \cdot y_j = 0 \ \forall \{i,j\} \in E\}) = \conv(\{\pmb{y} \in \{0,1\}^n \mid y_i + y_j \leq 1 \ \forall \{i,j\} \in E\})$.
\end{proposition}

\begin{proof}
To show the second equation, it suffices to show that the set \eqref{eq: stable_set} has binary extreme points. 
Let $\pmb{y^*}$ be a fractional point of \eqref{eq: stable_set}, with $y^*_{i'} \in (0,1)$. 
Consider the two points $\pmb{y^0} = \pmb{y^*} - y^*_{i'} \cdot \pmb{e^{i'}}$ and $\pmb{y^1} = \pmb{y^*} + (1-y^*_{i'}) \cdot \pmb{e^{i'}}$, where $\pmb{e^{i'}}$ denotes the unit vector with $i'$th-component being 1. 
Since $y^*_{i'}>0$, by complementarity constraints, we know $y^*_{j} = 0$ for any $j \in \delta(i')$. Therefore, it is easy to verify that both $\pmb{y^0}$ and $\pmb{y^1}$ satisfy the complementarity constraints $y_i \cdot y_j = 0 \ \forall \{i,j\} \in E$. 
Since $\pmb{y^*}$ can be written as the convex combination of $\pmb{y^0}$ and $\pmb{y^1}$, we know that the fractional point $\pmb{y^*}$ cannot be an extreme point of set \eqref{eq: stable_set}.

To prove the first equation, by Section~\ref{prop: binary_extreme}, we only need to show that, for any binary $(\pmb{y^*}, \pmb{q^*}, \pmb{\bar q^*}, \pmb{v^*}, \pmb{\bar v^*}) \in \S_{\T}$ and any $(i',j') \in E$, we have $y^*_{i'} \cdot y^*_{j'} = 0$. Since $\T$ is a feasible cover partition, there exists $T^* \in \T$ such that $i' \in T^*$ or $j' \in T^*$. 
Without loss of generality assume $i' \in T^*$.
We then obtain $j' \in \delta(T^*)$. 
By equations \eqref{eq: balas 3}, \eqref{eq: balas 6}, \eqref{eq: bilinear2} with respect to index $T^*$, we have: $\bar q^*_{T^*} \cdot y^*_i = \bar v^*_{T^*, i} = 0 \ \forall i \in T^*$, $q^*_{T^*} \cdot y^*_j = v^*_{T^*, j} = 0 \ \forall j \in \delta(T^*)$, $q^*_{T^*} + \bar q^*_{T^*} = 1$. 
Since $q^*_{T^*}, \bar q^*_{T^*} \in \{0,1\}$ and $i' \in T^*, j' \in \delta(T^*)$, at least one of $y^*_{i'}$ and $y^*_{j'}$ is 0. 
Thus we complete the proof.
\end{proof}

From Section~\ref{prop: stableset}, we directly obtain the next corollary.
\begin{corollary}
\label{cor: stablecut}
Let $G$ be a graph and let $\T$ be a feasible cover partition of $G$. 
If $\pmb{\alpha}^\transp \pmb{y} \leq \beta$ is a valid inequality for the stable set polytope of $G$, then for any $T \in \T$, $\pmb{\alpha}^\transp \pmb{v_{T}} \leq \beta q_T$ and $\pmb{\alpha}^\transp \pmb{\bar v_{T}} \leq \beta \bar q_T$ are both valid inequalities for $\conv(\S_{\T})$.
\end{corollary}

\begin{proof}
Arbitrarily pick a point $(\pmb{y^*}, \pmb{q^*}, \pmb{\bar q^*}, \pmb{v^*}, \pmb{\bar v^*})$ of $\S_{\T}$. 
From Section~\ref{prop: stableset}, we know $\pmb{\alpha}^\transp \pmb{y^*} \leq \beta$. 
By multiplying both sides of such inequality with $q^*_T$ and $\bar q^*_T$ for any $T \in \T$, from equation \eqref{eq: bilinear2}, we conclude the proof.
\end{proof}

\section{Numerical experiments}
\label{sec: experiments}

In this section, we conduct numerical experiments to compare the optimality gaps of various relaxations for some particular LPCCs.
We use Python 3.6.0 with CPLEX 12.9 as the LP and MIP solver. 

For each LPCC instance we created, its optimal value $v^*$ is obtained from formulating it as a MIP. Specifically, for any complementarity constraint $y_i \cdot y_j = 0$, we add additional binary variables $z_i, z_j$ and enforce $z_i + z_j \leq 1$, $y_i \leq z_i$, $y_j \leq z_j$. 
Then the \emph{optimality gap} of one particular relaxation procedure is $|(v_R-v^*)/v^*|$, where $v_R$ denotes the optimal value of such relaxation. 
Here we take the absolute value, because our generated instances are not necessarily all maximization problems, and $val^*$ can also be negative.

In this section, we use the following acronyms to denote the optimality gap for each individual relaxation method:
\begin{itemize}
\item[] \emph{LP:} Optimality gap of the LP relaxation of the LPCC, where we simply ignore all the complementarity constraints.
\item[] \emph{B\&C:} When formulating the LPCC as a MIP, it denotes the optimality gap of the internal branch-and-cut procedure of CPLEX, where we only allow the branching to occur at the root node and turn off a list of generic-purpose cuts: \emph{disjunctive cuts, Gomory cuts, multi-commodity flow cuts, zero-half cuts, MIR cuts, lift-and-project cuts}, and \emph{flow cover cuts}. All the other problem-specific cuts, e.g., \emph{clique cuts, cover cuts, GUB cuts} are kept.
\item[] \emph{ER-ee:} Optimality gap of the edge-by-edge-based extended relaxation $\mathcal M$ in \eqref{eq: systemM}. 
\item[] \emph{ER-vc:} Optimality gap of the vertex-cover-based extended relaxation $\E_\T(P)$.
\item[] \emph{ER-vc-cuts:} Optimality gap of the relaxation obtained by adding some of the cutting-planes we introduced before (see Section~\ref{sec: numerical_cuts} for details) to the vertex-cover-based extended relaxation $\E_\T(P)$.
\end{itemize}

In particular, by comparing columns \emph{ER-vc} and \emph{ER-vc-cuts}, one is able to see the strength of those added cuts. 
Here we should mention that, due to the lack of an efficient separation method for those cuts, to show how much they can further tighten the extended relaxation, when computing \emph{ER-vc-cuts} we treated those cuts as additional linear constraints and added them altogether into the formulation of $\E_\T(P)$.
Next, we introduce the specific cutting-planes from Section~\ref{sec: cutting_plane} that were added into extended relaxations to obtain \emph{ER-vc-cuts}.

\subsection{Cutting-planes for experiments}
\label{sec: numerical_cuts}
\subsubsection{Stable set cuts}
\label{sec: stable_set_cuts}
The first type of cuts comes from the stable set polytope of $G$. In specific, we consider the following well-known inequalities for the stable set polytope:
\begin{enumerate}
\item \emph{Clique constraints}: $\sum_{i \in C} y_i \leq 1$ where $C$ is a clique. 
\item \emph{Odd cycle constraints}: $\sum_{i \in D} y_i \leq (|D|-1)/2$ where $D$ induces a chordless odd cycle.
\item \emph{Odd anti-cycle constraints}: $\sum_{i \in \bar D} y_i \leq 2$ where $\bar D$ induces the complement of a chordless odd cycle.
\end{enumerate}
See \cite{MR368749} and \cite{lovasz2003semidefinite} for a detailed introduction to these families of cuts.
By Section~\ref{prop: stableset}, all these inequalities are valid for $\conv(\S_{\T})$, hence can be used as cutting-planes to strenghten $\E_{\T}(P)$. 
By Section~\ref{cor: stablecut}, the corresponding inequalities involving $\pmb{v} (\text{or }\pmb{\bar v})$ and $\pmb{q} (\text{or } \pmb{\bar q})$ can also be added as cuts. 

For a given graph $G$, we used the Python package \emph{NetworkX} to create the list of maximal cliques, odd cycles and odd anti-cycles. 

\subsubsection{Clique q-cuts}

This type of cuts comes from Section~\ref{prop: simple_qcut}: If clique $C$ is covered by a minimal subset $\{T_1, \ldots, T_k\}$ of $\T$, then adding $\sum_{i=1}^k q_{T_i} \leq 1$ into $\tilde{\E}_\T(P)$ remains to be an extended formulation of $P^\perp$, so it can be added to tighten the linear system of $\E_{\T}(P)$. 
Such inequality comes from a clique of the conflict graph and it only involves $q$-variables, hence we name it \emph{clique q-cut}. 
Similarly, inequality $\sum_{i=1}^k \pmb{v_{T_i}} \leq \pmb{y}$ can also be added into $\E_{\T}(P)$ to tighten the linear relaxation. 

\subsubsection{BQP cuts}
\label{sec: bqp_cuts}
The last type of cuts we added comes from $\QP_{2|\T|, n}$. As recently studied by Sripratak et al. \cite{sripratak2021bipartite}, a few families of valid inequalities have been proposed. However, the \emph{odd-cycle inequalities} 
%(do not mean the same as those for stable set polytope) 
are the only inequalities that can be facet-defining for $\QP_{2|\T|, n}$, while all the other inequalities therein, e.g., \emph{rounded clique inequalities, rounded cut inequalities} are actually valid for the linear relaxation given by the McCormick constraints.
Due to the lack of a complete understanding of the polyhedral structure of $\QP_{2|\T|, n}$, we only added \emph{odd-cycle inequalities} with a particular size into our formulation.
Specifically, we added the ones corresponding to a cycle $C$ with $|C| = 4$ and the edge subset $M$ with $|M| = 1$ in equation (24) in \cite{MR1017216}.
For any $T_1 \neq T_2 \in \T$ and $i \neq j \in [n]$, we obtain odd-cycle inequalities:
\begin{align*}
& -q_{T_2} - y_j + v_{T_1, j} + v_{T_2, i} + v_{T_2, j} - v_{T_1, i} \leq 0,\\
& - \bar q_{T_2} - y_j + \bar v_{T_1, j} + \bar v_{T_2, i} + \bar v_{T_2, j} - \bar v_{T_1, i} \leq 0,\\
& -q_{T_2} - y_j + \bar v_{T_1, j} + v_{T_2, i} + v_{T_2, j} - \bar v_{T_1, i} \leq 0,\\
& -\bar q_{T_2} - y_j +  v_{T_1, j} + \bar v_{T_2, i} + \bar v_{T_2, j} -  v_{T_1, i} \leq 0.
\end{align*}
Similarly, by picking a cycle $C$ with $|C| = 4$ and edge subset $M$ with $|M| = 3$, for any $T_1 \neq T_2 \in \T$ and $i \neq j \in [n]$, we obtain odd-cycle inequalities:
\begin{align*}
& q_{T_1} + y_j + v_{T_2, i} - v_{T_1, i} - v_{T_1, j} - v_{T_2, j} \leq 1,\\
& \bar q_{T_1} + y_j + \bar v_{T_2, i} - \bar v_{T_1, i} - \bar v_{T_1, j} - \bar v_{T_2, j} \leq 1,\\
& q_{T_1} + y_j + \bar v_{T_2, i} - v_{T_1, i} - v_{T_1, j} - \bar v_{T_2, j} \leq 1,\\
& \bar q_{T_1} + y_j + v_{T_2, i} - \bar v_{T_1, i} - \bar v_{T_1, j} - v_{T_2, j} \leq 1.
\end{align*}

We mention that the number of inequalities of the above form is $\mathcal O(|\T|^2 n^2)$ which is also $\mathcal O(n^4)$.
Without an efficient separation heuristic, adding these inequalities can be a great burden to the solver. 
Therefore, for this type of cuts, instead of adding all of them into the relaxation, we adopt the following iterative procedure: for each iteration and the optimal solution $p^*$ obtained at current step, we only add the BQP cuts that are able to separate $p^*$, and then solve the new relaxation and obtain a new optimal solution for the next iteration. We terminate the iterations when the optimal value $v'$ obtained in the current iteration is within a small distance $(<1\% \cdot v')$ from the optimal value in the last iteration, or when the iteration number exceeds a certain threshold (we pick 5 as the iteration upper bound for our later experiments).

\subsection{Instance sets and experimental results}
\label{sec: instance_sets}

In the following, we describe three applications from which we generated instances for our numerical experiments.
For the first two types of instances in Section~\ref{sec: tpesc} and Section~\ref{sec: cmkpc}, the corresponding conflict graphs were generated randomly according to the \emph{Erd\H{o}s-R\'{e}nyi model} with respect to some pre-determined density $\rho$, where each edge is included in the conflict graph with probability $\rho$, independently from every other edge. 
This can be realized using function \emph{erdos\_renyi\_graph} of Python package \emph{NetworkX}. 
For each of these conflict graphs, the function \emph{min\_weighted\_vertex\_cover} of \emph{NetworkX} will also produce an approximate minimum vertex cover, from which we produced the feasible cover partition to construct our extended formulation $\E_\T(P)$.
For the third type of instances in Section~\ref{sec: lpcc1regular}, we specifically considered the LPCC with 1-regular conflict graph, which has been studied extensively in the literature.

\subsubsection{Transportation problem with exclusionary side constraints}
\label{sec: tpesc}
The first type of instances we considered are instances of the \emph{Transportation Problem with Exclusionary Side Constraints} (TPESC), which was originally proposed by Cao \cite{MR1155839} and later also studied by Sun \cite{sun1998tabu}, and by Syarif and Gen \cite{syarif2003solving}. 
In this problem the goal is to transport goods from several sources $i \in S$ to destinations $j \in D$ at the lowest cost, while meeting all constraints on the supplies and demands.
Due to some specific reasons, in this problem it is not allowed to carry certain goods to the same destination.
For instance, hazardous materials may not be stored in the same warehouse.
Denote the source-destination pairs $L:= S \times D$, and the flow on arc $\ell \in L$ is denoted by $y_\ell$. If $\ell = \{i,j\}$ then $y_\ell$ is also denoted as $y_{i,j}$. Here we assume that $0 \leq y_\ell \leq 1$ for any $\ell \in L$.
Then TPESC can be modeled as a classical transportation problem with additional complementarity constraints:
\begin{align}
\label{TPESC}
\tag{TPESC}
\begin{split}
\text{minimize} \quad & \sum_{\ell \in L} c_\ell y_\ell  \\
\text{subject to} \quad & \sum_{j \in D} y_{i,j} = \alpha_i, \quad \forall i \in S \\
& \sum_{i \in S} y_{i,j} = \beta_j, \quad \forall j \in D \\
&  y_{i,j} \cdot y_{i',j} = 0, \quad \forall (i,i') \in E, j \in D \\
& 0 \leq \pmb{y} \leq \textbf{1}_{|L|}.
\end{split}
\end{align}
We remark that \eqref{TPESC} is a minimization problem instead of a maximization problem as \eqref{LPCC}.
Here $\pmb{\alpha} \in \R_+^{|S|}$ denotes the supplies, $\pmb{\beta} \in \R_+^{|D|}$ denotes the demands and $\pmb{c} \in \R_+^{|L|}$ denotes the transportation costs. 
We further assume that $\sum_{i \in S} \alpha_i = \sum_{j \in D} \beta_j$ and $\alpha_i \leq |D|$, $\beta_j \leq |S|$ for any $i \in S, j \in D$, since otherwise \eqref{TPESC} is infeasible. 
Let $E \subseteq S \times S$ denote the set of conflicting pairs of $S$ which cannot transport goods to the same destination in $D$. 
In other words, the conflict graph of \eqref{TPESC} is given by
$$
\big(L, \{\big( \{i,j\}, \{i', j\} \big) \in L \times L \mid (i, i') \in E, j \in D\} \big).
$$

For fixed source set $S$ and destination set $D$, we randomly generated our instances in the following way: $c_\ell$ was chosen uniformly at random from $[3,8]$ for any $\ell \in L$, $\alpha_i$ was chosen uniformly at random from $[1, \lfloor d/2 \rfloor]$ for any $i \in S$, and $\beta_j$ was chosen uniformly at random from $[1, \lfloor s/2 \rfloor]$ for any $j \in D$. 
In case supply exceeds demand, i.e., $\sum_{i \in S}\alpha_i > \sum_{j \in D}\beta_j$, then $\alpha_i$ values were evenly reduced until the sums are balanced. 
We analogously treated the $\beta_j$ values in case demand exceeds supply, i.e., 
$\sum_{i \in S}\alpha_i < \sum_{j \in D}\beta_j$.

Notice that for TPESC, the number of variables in the extended relaxation $\E_\T(P)$ is  $\mathcal O(|S|^2|D|^2)$.
For the ease of computation and illustration, we implemented our experiments using instances with relatively small sizes. 
In specific, we considered three problem sizes: $|S| = 10, |D| = 10$; $|S| = 30, |D| = 10$ and $|S| = 50, |D| = 5$. For each problem size, we also considered various conflict graph densities, and 10 randomly generated instances were solved with respect to each fixed $|S|, |D|$ and conflict graph density $\rho$.
For each relaxations mentioned before, the averaged numerical results are recorded in Section~\ref{smalltable0}. 

\begin{table}[H]
\centering
\caption{Optimality gap of different relaxations for \eqref{TPESC}. (\%)}
 \vspace{0.3cm}
\begin{tabular}{  p{2.4cm} | p{1cm} | p{1cm} | p{1.5cm} p{1.5cm} | p{2.1cm}  }
\hline
$(|S|, |D|, \rho)$ & LP   & B\&C & ER-ee & ER-vc &  ER-vc-cuts    \\ 
\midrule
(10, 10, 0.2) &  4.98   &  0.06  &  0.48  & 0.43  &  0.00\\
(10, 10, 0.4) &  11.35   &  0.10  & 3.03   & 2.14  & 0.00    \\
(10, 10, 0.6) &  12.45   &  0.85  & 3.60   & 1.72  & 0.00    \\
\midrule
(30, 10, 0.04) &  3.33   &  0.17  &  0.18  &  0.11  &  0.00    \\
(30, 10, 0.1) &  17.7   &  1.84  &  3.24  &  3.24  &  0.00    \\
(30, 10, 0.2) &  28.95    &  8.22  & 12.57  & 11.43 &  0.67   \\
\midrule
(50, 5, 0.04) & 8.40 &  0.31  &  0.31    & 0.27  & 0.02 \\
(50, 5, 0.08) &  15.68   &  2.16  & 3.58   & 3.58  & 0.00  \\
(50, 5, 0.12) &  20.99   &  5.85  & 8.01   & 8.01  & 2.86  \\
\bottomrule
 \end{tabular}
 \label{smalltable0}
\end{table}

The first thing that one might notice is that the optimality gaps of almost all relaxations increase when the conflict graph becomes denser.
Also, as shown by the smaller values in column \emph{ER-vc} than those in \emph{ER-ee}, we know that the vertex-cover-based formulation $\E_\T(P)$ is generally a tighter relaxation than the edge-by-edge formulation $\mathcal{M}$. This observation justifies our statement at the end of Section~\ref{sec: relaxation_edge_by_edge}. 
Moreover, even though in most of the testing results, the extended relaxation $\E_\T(P)$ itself alone does not seem to provide better optimality gap than the relaxation of \emph{B$\&$C}, but adding cutting-planes into the extended relaxation is able to provide a huge improvement for closing the optimality gap. This phenomenon will be further verified by the following experiments.

\subsubsection{Continuous multi-dimensional knapsack problem with conflicts}
\label{sec: cmkpc}

The second type of instances that we considered comes from the \emph{Continuous Multi-dimensional Knapsack Problem with Conflicts} (CMKPC):
 \begin{align}
\label{CMKPC}
\tag{CMKPC}
\begin{split}
\text{maximize} \quad & \pmb{c}^\transp \pmb{y}  \\
\text{subject to} \quad & M \pmb{y} \leq \pmb{b}, \\
&  y_{i} \cdot y_{j} = 0, \quad \forall \{i,j\} \in E \\
& 0 \leq \pmb{y} \leq \textbf{1}_{n}.
\end{split}
\end{align}
Here $M \in \Z_+^{m \times n}, \pmb{b} \in \Z_+^m$ and $\pmb{c} \in \Z_+^n$. 

The special case of CMKPC where there is only a single knapsack constraint and the conflict graph $([n], E)$ is given by the union of some disjoint cliques, was first investigated by Ibaraki et al. \cite{MR568622, ibaraki1978multiple}.
Later, de Farias et al. \cite{de2002facets} studied inequalities that define facets of the convex hull of its feasible set.
Later, Hifi and Michrafy \cite{hifi2007reduction}, Pferschy and Schauer \cite{pferschy2009knapsack}, Luiz et al. \cite{MR4330950} dealt with the knapsack problem with general conflict graph, but with binary variables.

Within the implementation of the experiments over CMKPC, we specifically considered three problem sizes: $n=20, m=5$; $n=60, m=4$ and $n=100, m=2$. For each problem size we randomly generated conflict graph with respect to various graph densities. Then for each problem size and graph density, we conducted numerical experiments over 10 randomly generated instances, where we adopted the same setup as in \cite{MR3773088}: the objective coefficients $c_j$ and row coefficients $M_{i,j}$ were generated uniformly at random in the interval $[10, 25]$, and the right hand side coefficients $b_i$ were determined by $b_i = 0.3 \cdot \textbf{1}_n^\transp M_{i}$, which comes from de Farias and Nemhauser \cite{MR1993458}. 
We report the averaged numerical results of those 10 experiments in Section~\ref{smalltable1}.

\begin{table}[ht]
\centering
\caption{Optimality gap of different relaxations for \eqref{CMKPC}. (\%)}
 \vspace{0.3cm}
\begin{tabular}{  p{2.4cm} | p{1.2cm} | p{1.2cm} | p{1.5cm} p{1.5cm} | p{2.1cm}  }
\hline
$(n,m, \rho)$ & LP   & B\&C & ER-ee & ER-vc &  ER-vc-cuts    \\ 
\midrule
(20, 5, 0.05) &  2.05   &  0.23  &  0.01  &  0.01  &      0.00   \\
(20, 5, 0.1) &  4.89   &  0.69  &  0.18  &  0.15  &      0.00    \\
(20, 5, 0.35) &  17.82   &  1.89  & 6.83   & 5.18   &  0.00    \\
(20, 5, 0.6) &  57.35   &  0.65  &  41.61  &  29.27  &  0.00   \\
\midrule
(60, 4, 0.05) & 4.37   &  0.23  & 0.26   & 0.23   &  0.00   \\
(60, 4, 0.1) & 10.37   &  0.38  & 0.82   & 0.79   &  0.02  \\
(60, 4, 0.15) & 25.90   &  5.49  & 11.96   & 11.96   &  0.76  \\
(60, 4, 0.2) & 41.90   & 11.68 & 26.51 & 25.60 & 1.78 \\
(60, 4, 0.35) &  102.17  &  15.72  & 79.01   &  77.78  &  3.73    \\
(60, 4, 0.6) & 220.51   &  21.21  &  187.36  &  161.09  &  3.90      \\
\midrule
(100, 2, 0.05) & 13.68   & 0.18  &  0.91  &  0.91 &  0.00   \\ 
(100, 2, 0.1) & 29.24   & 6.29  &  10.99  &  10.99 &  1.07   \\ 
(100, 2, 0.15) &  54.78  &  20.57  &  40.58  & 40.57  &   8.14 \\ 
(100, 2, 0.2) &  93.45   & 22.76  &  70.12   &  69.90  & 9.52 \\ 
% (100, 2, 0.35) &  192.25  & 32.09  &  150.54  & 150.13  &  13.07  \\
% (100, 2, 0.6) & 363.92   & 34.26  & 310.68   & 273.78  &  \textbf{17.67}   \\
\bottomrule
 \end{tabular}
 \label{smalltable1}
\end{table}

The numerical results in Section~\ref{smalltable1} further illustrate the great performance of our proposed extended relaxation when enhanced with cutting-planes. 
One can also see that, when the conflict graph is sparse, the extended relaxation from either the edge-by-edge formulation $\mathcal M$ or the vertex-cover-based formulation $\E_\T(P)$ can actually provide a decent approximation value for the corresponding LPCC. 
However, when the conflict graph becomes denser, those additional cutting-planes are able to dramatically further close the optimality gap.

\subsubsection{LPCC with 1-regular conflict graph}
\label{sec: lpcc1regular}

For the last type of instances, we considered the special case of LPCC with 1-regular conflict graph, which is the problem studied in Sect.~4.2 of \cite{MR4270189}:
\begin{align}
\label{1R}
\tag{1R}
\begin{split}
\text{maximize} \quad & \pmb{f_x}^\transp \pmb{x} + \pmb{f_y}^\transp \pmb{y} + \pmb{f_z}^\transp \pmb{z}  \\
\text{subject to} \quad & A \pmb{x} + B \pmb{y} + C \pmb{z} \leq \pmb{d},\\
&  y_{i} \cdot z_{i} = 0, \quad \forall i \in [n] \\
& 0 \leq \pmb{x} \leq \textbf{1}_{p}, \ 0 \leq \pmb{y}, \pmb{z} \leq \textbf{1}_n.
\end{split}
\end{align}
Here $p, n \in \N$, $\pmb{f_x} \in \R^p$, $\pmb{f_y}, \pmb{f_z} \in \R^n$, $A \in \R^{m \times p}$, $B, C \in \R^{m \times n}$, $\pmb{d} \in \R^m$. 

We generated 25 instances of \eqref{1R} for each set of problem size using the same setup as in \cite{MR4270189} (Sect.~4.2). In detail:
\begin{enumerate}
\item The entries of vectors $\pmb{f_x}, \pmb{f_y}, \pmb{f_z}$ are integers generated uniformly at random between -15 and 5;
\item The entries of $A$ and $B$ are integers generated uniformly at random between -20 and 30;
\item The entries of $C$ are obtained by generating integers $U_{i,j}$ uniformly at random between -10 and 10, and then by computing $C_{i,j} = -B_{i,j} + (\sign(B_{i,j}) + 0.5)^{-1} U_{i,j}$;
\item The entries of $\pmb{d}$ are obtained by generating reals $\theta_j$ uniformly at random between 0 and 1, and then by computing $d_i = \lfloor \theta (B_{i} + C_{i}) \textbf{1}_n \rfloor$. 
\end{enumerate}
Unlike in \cite{MR4270189}, where the authors only considered problem size $n = 6, p = 10, m = 10$, here we also considered another two sets of problem size: $n = 20, p = 10, m = 5$ and $n = 20, p = 20, m = 20$.

For LPCC with 1-regular conflict graph, when constructing the vertex-cover-based extended relaxation, we chose $\{p+1, \ldots, p+n\}$ as the vertex cover, which are the set of indices correspond to the $\pmb{y}$-components. 
It is not hard to observe that, in this case, the edge-by-edge extended relaxation \eqref{eq: systemM} coincides with such vertex-cover-based extended relaxation. Therefore, we do not report results of \emph{ER-ee}. Furthermore, when the conflict graph is 1-regular, it will not contain any odd cycle, odd anti-cycle, or clique with size larger than 2, so for \emph{ER-vc-cuts}, the cutting-planes we added into $\E_\T(P)$ only came from the \emph{BQP cuts} in Section~\ref{sec: bqp_cuts}.
We report the detailed results in Section~\ref{smalltable2}. 

\begin{table}[H]
\centering
\caption{Optimality gap of different relaxations for \eqref{1R}. (\%)}
 \vspace{0.3cm}
\begin{tabular}{  p{2.4cm} | p{1.5cm} | p{1.5cm} | p{1.5cm} | p{2.1cm}  }
\hline
$(n, p, m)$ & LP   & B\&C &  ER-vc &  ER-vc-cuts    \\ 
\midrule
(6, 10, 10) &  30.41   &  16.27  &  0.83  &       0.70    \\
\midrule
(20, 10, 5) &  13.64   &  3.36  &  0.29  &       0.28    \\
\midrule
(20, 20, 20) &  322.87    &  98.29   &  0.60   &    0.46      \\
\bottomrule
 \end{tabular}
 \label{smalltable2}
\end{table}

As we have argued in Section~\ref{sec: 1-r}, for LPCC with 1-regular conflict graph, our proposed extended relaxation coincides with the ERLT relaxation in \cite{MR4270189}.
For this class of instances, the above Section~\ref{smalltable2} demonstrates the great performance of such extended relaxation, which confirms that ERLT relaxation can be significantly stronger than those available in literature, as remarked in \cite{MR4270189}.
Notice that when the conflict graph is 1-regular, the density of the edges is simply $n/\binom{2n}{2}=1/(2n-1)$, which means that 1-regular conflict graphs are quite sparse. As we have also observed in Section \ref{sec: cmkpc}, our extended relaxation can generally provide much better results when the conflict graph is sparse. Therefore, in some sense, this justifies the great performance of ERLT for this type of problems. Moreover, the last column of \emph{ER-vc-cuts} shows that, despite the great performance of ERLT relaxation as shown in column \emph{ER-vc}, additional \emph{BQP cuts} can still further close the optimality gaps.

\section{Conclusion}
\label{sec: conclude}

In this paper, we studied linear relaxations in an extended space for linear programs with general complementarity constraints, which generalizes the recent convexification technique called ERLT developed by Nguyen et al. \cite{MR4270189}. 
We then proposed a few classes of cutting-planes to further strengthen the linear relaxation.
We presented numerical experiments that showcase the significant improvement obtained from the addition of those cuts. 
To facilitate the future practical uses, one interesting future direction is the study of efficient separation algorithms for those cuts in the extended space.
We believe that this paper provides a new perspective for developing branch-and-cut approaches for solving LPCCs with general conflict graph.

\bibliographystyle{plain}
\bibliography{biblio}

\end{document}